\documentclass[12pt]{amsart}
\usepackage{amsmath}
\usepackage{mathtools,verbatim}
\usepackage{amsthm}
\usepackage{amssymb}
\usepackage{mathrsfs}
\usepackage{color}

\allowdisplaybreaks

\numberwithin{equation}{section}

\newcommand*{\Norm}[2]{\left\lVert {#1} \right\rVert_{#2}}
\newcommand{\f}{\frac}
\newcommand{\al}{\alpha}
\newcommand{\be}{\beta}
\newcommand{\ga}{\gamma}
\newcommand{\nf}{\infty}
\newcommand{\ve}{\varepsilon}
\newcommand{\q}{\quad}
\newcommand{\rn}{\mathbb R^n}
\newcommand{\C}{\mathbb C}
\newcommand{\R}{\mathbb R}

\newcommand{\om}{\omega}

\newcommand{\qq}{\qquad}
\newcommand{\bee}{\begin{equation}}
\newcommand{\eee}{\end{equation}}
\newcommand{\lab}{\label}
\newcommand{\p}{\partial}
\newcommand{\wh}{\widehat}

\allowdisplaybreaks

\newtheorem{Theorem}{Theorem}[section]
\newtheorem{Corollary}[Theorem]{Corollary}
\newtheorem{Lemma}[Theorem]{Lemma}
\newtheorem{Prop}[Theorem]{Proposition}

\newtheorem{Conjecture}[Theorem]{Conjecture}

\theoremstyle{definition}
\newtheorem{Remark}[Theorem]{Remark}
\newtheorem{Example}[Theorem]{Example}

\begin{document}

\title[Marcinkiewicz Multiplier Theorem]{ 
A  Sharp Variant of the  Marcinkiewicz  Theorem   with  
Multipliers in Sobolev Spaces of Lorentz type}

\author{Loukas Grafakos}

\address{Department of Mathematics, University of Missouri, Columbia MO 65211, USA}
\email{grafakosl@missouri.edu}

\author{Mieczys\l aw Masty\l o}

\address{Faculty of Mathematics and Computer Science,
Adam~Mickiewicz University, Pozna\'n, Uniwersytetu Pozna\'nskiego 4,
61-614 Pozna{\'n}, Poland}
\email{mastylo$@$math.amu.edu.pl}

\author{Lenka Slav\'ikov\'a}

\address{Mathematical Institute, University of Bonn, Endenicher Allee 60,
53115 Bonn, Germany}
\email{slavikova@karlin.mff.cuni.cz}

\thanks{The first author acknowledges the support of the Simons Foundation. The second author was supported by the National
Science Centre, Poland, Grant no. 2019/33/B/ST1/00165. The third author was supported by Deutsche Forschungsgemeinschaft
(DFG, German Research Foundation) under Germany's Excellence Strategy - GZ 2047/1, Projekt-ID 390685813.}

\begin{abstract}
Given a bounded  measurable function $\sigma$ on $\mathbb{R}^n$, we  let $T_\sigma $ be the operator obtained
by multiplication  on the Fourier transform by $\sigma $. Let $0<s_1\le s_2\le  \cdots \le s_n<1$ and $\psi$ be a Schwartz
function on the real line whose Fourier transform $\widehat{\psi}$ is supported in  $[-2,-1/2]\cup[1/2,2]$  and which satisfies
$\sum_{j \in \mathbb{Z}} \widehat{\psi}\left(2^{-j} \xi\right)=1$ for all $\xi \neq 0$. In this work we sharpen the known
forms of the Marcinkiewicz multiplier theorem by finding an almost optimal function space with the property that, if the
function
\begin{equation*}
(\xi_1,\dots, \xi_n)\mapsto	  \prod_{i=1}^n (I-\partial_i^2)^{\frac {s_i}2}
\Big[ \prod_{i=1}^n \widehat{\psi}(\xi_i) \sigma(2^{j_1}\xi_1,\dots , 2^{j_n}\xi_n)\Big]
\end{equation*}
 belongs to it uniformly in $j_1,\dots , j_n \in \mathbb Z$, then $T_{\sigma}$ is bounded on
$ {L}^p(\mathbb R^n)$ when $ |\frac{1}{p}-\frac{1}{2}  | < s_1$ and $1<p<\infty$.
In the case where $s_i\neq s_{i+1}$ for all $i$,   it was proved
in \cite{Paper1} that the   Lorentz space $L ^{\frac{1}{s_1},1} (\mathbb{R}^n) $ is the  function space sought.
In this work we address the    significantly more difficult general case
when for certain indices $i$
we might have $s_i=s_{i+1}$. We obtain a version of the Marcinkiewicz multiplier theorem in which
the space     $L ^{\frac{1}{s_1},1}$
 is replaced by an
appropriate Lorentz space
associated with a certain concave function related to the number of terms among $s_2,\dots , s_n$ that
equal $s_1$. Our result is optimal up to an 
arbitrarily small power of the logarithm  in the defining concave function of the Lorentz space. 
\end{abstract}

\maketitle

\section{Introduction}\label{S:introduction}

Let $\mathscr C_0^\infty(\rn)$ be the space of smooth functions with compact support on $\rn$.  Given any function
$\sigma$ in ${L}^{\infty}(\mathbb{R}^n)$, we consider the multiplier operator $T_\sigma$ defined for all
$f\in  \mathscr C_0^\infty(\rn)$ by
\begin{align*}
T_{\sigma}f (x)= \int_{\mathbb{R}^n} \widehat{f}(\xi) \sigma(\xi) {e}^{2\pi i x \cdot \xi}d\xi, \quad\, x\in \mathbb{R}^{n}.
\end{align*}
As usual, here and in the sequel, $\widehat{f}$ denotes the Fourier transform of $f$ given by
\[
\widehat{f}(\xi)= \int_{\mathbb{R}^n}f(x) {e}^{-2\pi i x \cdot \xi}\,dx, \quad\, \xi\in \rn.
\]
The theory of multipliers is vast and extensive but basic material about them can be found in   \cite{Hor}, \cite{CFA14} and
    \cite{Larsen}.

A classical problem in harmonic analysis is to find good sufficient conditions on functions $\sigma$ guaranteeing that
    $T_{\sigma}$ extends to a~bounded operator on $L^p(\mathbb{R}^n)$ for some $1<p<\infty$.  If this is the case,
then $\sigma$ is called an $L^p$ Fourier multiplier. This  problem has a~long history going back to Bernstein,
Hardy, Weyl, Marcinkiewicz, Mikhlin and was
studied in the sixties  by several mathematicians including   Calder\'on \cite{Cal}, Hirschman \cite{Hirsch},
H\"ormander  \cite{Hor}, de Leeuw \cite{Leeuw}, Carleson and Sj\"olin \cite{CaS}.

The significance of the   multiplier problem lies in the fact that many classical
$L^p$ boundedness problems in analysis can be described
in terms of Fourier multipliers. Several conditions on $\sigma$ are known to imply boundedness
for $T_\sigma$ on $L^p(\mathbb{R}^n)$.
We are not going into a~complete historical overview of   multiplier theory, but we focus on versions of the
Marcinkiewicz multiplier theorem. We
start with the classical result of Marcinkiewicz~\cite{Marc}, first proved in the context of two-dimensional Fourier series, which basically says (in $n$ dimensions) that
if for all $\al_j \in \{0,1\}$ \bee\lab{Eq-Marc2} \Big| \p_{\xi_1}^{\alpha_1} \cdots  \p_{\xi_n}^{\alpha_n} m (\xi) \Big|
\le C_\al |\xi_1|^{ -\al_1}\cdots |\xi_n|^{ -\al_n} , \qq \textup{$\xi_j\neq 0$ when $\al_j=1$}, \eee then $T_m$ is bounded on $L^p(\rn)$
for all $p\in (1,\infty)$.

In order to fine-tune this theorem we discuss a version of it where the derivatives  $\al_j$ could be fractional. To describe this
we introduce a~Schwartz function $\psi$ on $\mathbb{R}$ whose Fourier transform is supported
in $[-2,-1/2]\cup [1/2,2]$ and which satisfies $\sum_{j\in \mathbb Z} \widehat{\psi}(2^{-j}\xi)=1$ for all $\xi\neq 0$. We then define
a~function $\Psi$ on $\rn$ such that
\bee\label{defPsi}
\wh{\Psi} =  \underbrace{\widehat{\psi}\otimes \cdots \otimes \widehat{\psi} }_ {\textup{$n$ times }}\,.
\eee
Here,
\[
(g_1\otimes \cdots \otimes g_n ) (x_1,\dots, x_n) := g_1(x_1) \cdots g_n(x_n), \quad\, (x_1,\dots, x_n)\in \mathbb{R}^{n}\,,
\]
stands for the tensor product of functions $g_j\colon \R \to \C$, $1\leq j \leq n$.
We use the following notation for the differential operator
\[
\Gamma(s_1,\dots , s_n):= (I-\partial_1^2)^{s_1/2} \cdots (I-\partial_n^2)^{s_n/2},
\]
where $\p_j$ denotes differentiation in the $j$th variable.  We also introduce the multi-dilation operator
\[
D_{j_1,\dots , j_n} g (\xi_1, \dots , \xi_n) :=  g (2^{j_1} \xi_1, \dots , 2^{j_n}\xi_n), \quad\, (\xi_1, \dots , \xi_n) \in \mathbb{R}^{n}\,,
\]
where $g$ is a function on $\rn$ and $j_1,\dots, j_n \in \mathbb Z$.

When $0< 1/r<s_1\le \cdots \le s_n<1$, it was shown in \cite{GSMarc19} that if
\begin{equation}\lab{FC99}
\sup_{j_1, \dots , j_n \in \mathbb{Z}}\Big\|\Gamma(s_1,\dots,s_n) \big[\wh{\Psi}
D_{j_1,\dots , j_n} \sigma \big]\Big\|_{L^r(\mathbb{R}^n)}  < \infty\,,
\end{equation}
then $T_{\sigma}$ maps ${L}^p(\mathbb R^n)$ to ${L}^p(\mathbb R^n)$ when $ | \f1p -\f12 | < s_1$. Earlier versions
of this result were provided by Carbery~\cite{C}, who   considered the case in which the multiplier lies in a~product-type
$L^2$-based Sobolev space, and  Carbery and Seeger~\cite[Remark after Prop.\ 6.1]{CSTAMS}, who considered the case
$s_1=\cdots = s_n > | \f1p -\f12 |=\f1r$. The positive direction of Carbery and Seeger's result in the range
$| \f1p -\f12 | < \f1r$ also  appeared in~\cite[Condition (1.4)]{CS}; note that in these cases the range is  expressed in
terms of the integrability of the  multiplier and not in terms of its smoothness.
Alternative improvements and variants of
the Marcinkiewcz multiplier theorem were also proved by Coifman, Rubio de Francia  and Semmes \cite{CFS} and Tao and Wright \cite{TW}.
An  extension of the Marcinkiewicz multiplier theorem to general Banach spaces was obtained by Hyt\"onen \cite{Hyt}.

A weakening of the condition in \eqref{FC99} was provided in   \cite{Paper1}, where the $L^r$ space was replaced by the locally
larger Lorentz space ${L}^{{1}/{s_1},1}(\mathbb R^n)$. But this was achieved under the additional hypothesis that $s_1<s_2<\cdots <s_n$; the case $n=2$ was first proved in \cite{Paper0}. 
In this paper we deal with the more complicated case when a streak of $s_j$'s could be identical. In this case, the Lorentz-space estimate from~\cite{Paper1} fails (see Example~\ref{example} below for the proof of this assertion). Nevertheless, we show that a~limiting version 
of the Marcinkiewicz multiplier theorem can still be obtained. We achieve this goal by enlarging the original Lorentz space 
${L}^{{1}/{s_1},1}(\mathbb R^n)$ by inserting in the defining function a certain power of the logarithm. The class of function spaces 
that is suitable for solving this problem is described below.
Given a concave function $\varphi: \mathbb R_+ \rightarrow \mathbb R_+$ that is positive on $(0,\infty)$ and satisfies $\varphi(0)=0$, 
we define the \emph{Lorentz space} $\Lambda_\varphi$ on $\mathbb{R}_{+}$ to be the space of all measurable functions $f$ on $\R_+$ for which
\[
\|f\|_{\Lambda_\varphi} := \int_0^\infty f^{*}(t)\,d\varphi(t)<\infty\,,
\]
where $f^{*}$ denotes the non-increasing rearrangement of $f$. If $1\leq p<\infty$ and $\varphi(t) = t^{1/p}$ for all $t\geq 0$, then we
recover the classical Lorentz space $L^{p, 1}$.

For $s\in (0, 1)$ and  $\beta \in \R$, we consider a concave function $\phi_{s, \be}$  such that
\begin{equation}\label{E:phi}
\phi_{s,\be}(t) \approx t^s \log^{\beta}\big(e+\tfrac{1}{t}\big), \quad t>0\,.
\end{equation}
The main result of this paper is the following theorem.

\begin{Theorem}
\label{1/2MainResult}
Let $1<p<\infty$ and $\Psi$ be as in \eqref{defPsi}. Let $0<s_1\le s_2 \le \cdots \le s_n <1$ and assume that there are exactly $d$
numbers among   $s_2,\dots, s_n$ that equal $s_1$. In addition, assume that $s_1>|1/p-1/2|$.  If a~function $\sigma \in L^{\infty}(\mathbb{R}^n)$ satisfies
\begin{equation}
\lab{hypoK}
K:=	\sup_{j_1,\dots , j_n\in \mathbb{Z} } \Norm{\Gamma \left(s_1,\dots , s_n\right)\big[  \wh{\Psi}\,
D_{j_1,\dots , j_n} \sigma \big]}{\Lambda_{\phi_{s_1, d}}(\mathbb{R}^n)} < \infty\,,
\end{equation}
then there is constant $C=C(s_1,\dots, s_n,p,n,d,\psi)$
such that, for every $f\in \mathscr C_0^\infty(\rn)$, we have
\begin{equation}\label{E:bound}
\Norm{T_\sigma f}{{L}^p(\mathbb{R}^n)} \leq C K \Norm{f}{{L}^p(\mathbb{R}^n)}\,.
\end{equation}
Thus, $T_{\sigma}$ admits a bounded extension from ${L}^p(\mathbb{R}^n)$ to ${L}^p(\mathbb{R}^n)$ with the same bound.
\end{Theorem}

Naturally, the theorem remains invariant under any permutation of the variables. It was only stated in the case
where the index $s_j$ corresponds to variable $\xi_j$ for simplicity. We also point out that the power $d$ of the
logarithm in condition~\eqref{hypoK} can be slightly lowered if we allow it to depend on $s_1$;   on this improvement see
Remark~\ref{R:remark_logarithm}.   
The results contained  in Theorem~\ref{1/2MainResult} and Remark~\ref{R:remark_logarithm} inspire  the following  speculation related to the  optimal power of the logarithm.

\begin{Conjecture}
\label{C:conjecture}
Let $p$, $\Psi$, $s_1,\dots,s_n$ be as in Theorem~\ref{1/2MainResult}. If a~function 
$\sigma$ in $L^{\infty}(\mathbb{R}^n)$ satisfies
\begin{equation}
\label{E:optimal_space}
K':=	\sup_{j_1,\dots , j_n\in \mathbb{Z} } \Norm{\Gamma \left(s_1,\dots , s_n\right)\big[  \wh{\Psi}\,
D_{j_1,\dots , j_n} \sigma \big]}{\Lambda_{\phi_{s_1, (1-s_1)d}}(\mathbb{R}^n)} < \infty\,,
\end{equation}
then there is constant $C=C(s_1,\dots, s_n,p,n,d,\psi)$
such that, for every $f\in \mathscr C_0^\infty(\rn)$, we have
\begin{equation*}
\Norm{T_\sigma f}{{L}^p(\mathbb{R}^n)} \leq C K' \Norm{f}{{L}^p(\mathbb{R}^n)}\,.
\end{equation*}
\end{Conjecture}

We recall that $\|T_\sigma\|_{L^p(\R^n) \rightarrow L^p(\R^n)} \gtrsim \|\sigma\|_{L^\infty(\R^n)}$, and
the appearance of the space $\Lambda_{\phi_{s_1, (1-s_1)d}}(\mathbb{R}^n)$ in Conjecture~\ref{C:conjecture} 
is motivated by the fact that  among all rearrangement-invariant spaces $E(\R^n)$ satisfying the Sobolev-type 
embedding
\[
\|\sigma\|_{L^\infty(\R^n)} \lesssim \|\Gamma(s_1,\dots,s_n) \sigma\|_{E(\R^n)},
\]
$\Lambda_{\phi_{s_1, (1-s_1)d}}(\mathbb{R}^n)$ is locally the largest one, see Proposition~\ref{P:optimality} 
below. Our emphasis on the local behavior of the function spaces involved when investigating optimality questions 
is then justified by the local nature of the condition~\eqref{hypoK}. We point out that while the validity of 
Conjecture~\ref{C:conjecture} remains an open problem, we are able to show that condition~\eqref{E:optimal_space} 
is sufficient for the $L^p$ boundedness of the operator $T_\sigma$ when the power $(1-s_1)d$ of the logarithm is 
replaced by $(1-s_1)d+\varepsilon$ for some $\varepsilon>0$, assuming that $s_1 \leq 1/2$; this is the content 
of Remark~\ref{R:remark_logarithm} below.

Throughout the paper we use standard notation. Given two nonnegative functions $f$ and $g$ defined on the same set
$A$, we write $f \lesssim g$, if there is a~constant $c>0$ such that $f(x) \leq c g(x)$ for all $x\in A$, while 
$f \approx g$ means that both $f \lesssim g$ and $g \lesssim f$ hold. If $X$ and $Y$ are Banach spaces, then 
$X\hookrightarrow Y$ means that $X \subset Y$ and the inclusion map is continuous. If $X$ and $Y$ are
Banach spaces, then we write $X=Y$ if $X\hookrightarrow Y$ and $Y\hookrightarrow X$. The measure space of all 
Lebesgue's measurable subsets of $\mathbb{R}^n$ equipped with Lebesgue measure $\lambda_n$ is denoted by 
$(\mathbb{R}^n, \lambda_n)$. For simplicity of notation, $\lambda$ denotes the Lebesgue measure restricted 
to Lebesgue's measurable subset of $\mathbb{R}_{+}:=[0, \infty)$. We use $C$ to describe an inessential
constant that may vary from occurrence to occurrence.

\section{Background   material}

Let $(\Omega, \mu):=(\Omega, \Sigma, \mu)$ be a~$\sigma$-finite measure space and let $L^0(\mu)$ denote the space 
of all (equivalence classes) of scalar valued (real or complex) $\Sigma$-measurable functions on $(\Omega, \mu)$ 
(on $\Omega$ for short) that are finite $\mu$-a.e. A~Banach space $X\subset L^0(\mu)$ is said to be a~Banach function 
space over $\Omega$ if  for all $f, g\in L^0(\mu)$ with $|g| \leq |f|$ $\mu$-a.e. and $f\in X$, one has $g\in X$ and 
$\|g\|_X \leq \|f\|_X$. The K\"othe dual space $X'$ of a Banach function space $X$ on $\Omega$ is a Banach function 
space of those $f\in L^0(\mu)$ for which $\|f\|_{X'} := \sup\big\{\int_\Omega |f g|\,d\mu:\, \|g\|_X \leq 1\big\}$ 
is finite.

Given $f\in L^0(\mu)$, its \emph{distribution function} is defined by $\mu_f(\tau)=\mu(\{x\in\Omega: \,|f(x)|> \tau\})$,
$\tau >0$, and its \emph{nonincreasing rearrangement} by $f^{*}(t)=\inf \{\tau \geq 0: \mu_f(\tau)\leq t\}$, $t\geq 0$.
A~Banach function space $E$ is called a~\emph{rearrangement-invariant} (r.i.) space if $\|f\|_E = \|g\|_E$ whenever
$\mu_f = \mu_g$ and $f\in E$.

Let $E$ be an~r.i.\ space on $\mathbb{R}_+$ and let $(\Omega, \mu)$ be a~measure space. Then we define the r.i.\ space 
$E(\Omega)$ on $\Omega$ to be the space of all $f\in L^0(\mu)$ such that $f^{*} \in E$ with $\|f\|_{E(\Omega)} = \|f^{*}\|_E$. 
Many properties of r.i.\ spaces can be expressed in terms of conditions on their Boyd indices. Recall that for any 
r.i.\ space $E$ on $\mathbb{R}_+$, we define the dilation operators $\sigma_s$ for $0 <s < \infty$ by
\[
\sigma_{s}f(t) = f(t/s), \quad\, f\in E, \,t\geq 0\,.
\]
Since $s\mapsto \|\sigma_s\|_E = \sup_{\|f\|_E \leq 1}\|\sigma_{s}f\|_E$ is a finite submultiplicative function on $(0, \infty)$,
the Boyd indices given by
\[
\alpha_E := \lim_{s\to 0+} \frac{\log \|\sigma_s\|_E}{\log s}, \quad\, \beta_E := \lim_{s\to \infty} \frac{\log \|\sigma_s\|_E}{\log s}
\]
are well defined and satisfy $0\leq \alpha_E \leq \beta_E \leq 1$ (see \cite[p.~99]{KPS})\,.

In the theory of operators on r.i. spaces the Lorentz and the Marcinkiewicz space play a~fundamental role. Let $\mathcal{P}$ be
the~set of functions $\varphi \colon \mathbb{R_+}\to \mathbb{R_+}$ that are concave, positive on $(0, \infty)$ and $\varphi(0) = 0$.

Given $\varphi\in\mathcal{P}$, the \emph{Lorentz space} $\Lambda_\varphi$ on $\mathbb{R}_{+}$ consists of all $f\in L^0(\lambda)$
such that
\[
\|f\|_{\Lambda_\varphi} := \int_0^\infty f^{*}(t) d\varphi(t) = \varphi(0+) f^{*}(0+)\, + \, \int_0^\infty f^{*}(t)\varphi'(t)\,dt\,,
\]
where $\varphi'$ is the derivative of $\varphi$, which exists except at a~countable set. We note that the functional $\|\cdot\|_{\Lambda_\varphi}$
induced by an increasing function $\varphi\colon \mathbb{R}_{+}\to \mathbb{R}_{+}$ is a~norm if and only if $\varphi$ is concave and
$\varphi(0) = 0$ \cite{Lorentz}. We note that $\Lambda_\varphi$ is a~separable space if and only if $\varphi(0+)=0$ and
$\varphi(+\infty) := \lim_{t\to \infty} \varphi(t) = \infty$.

Let $\mathcal{Q}$ be the set of all quasi-concave functions $\varphi\colon \mathbb{R}_{+} \to \mathbb{R}_{+}$, that is, of all positive functions
$\varphi$ on $(0, \infty)$ such that $\varphi(s) \leq \max\{1, s/t\} \varphi(t)$ for all $s, t>0$. Note that, for any $\varphi\in \mathcal{Q}$,
the function $\varphi_{*}$ given by $\varphi_{*}(t) := t/\varphi(t)$ for all $t>0$ is also a~quasi-concave function. We also note that for every
quasi-concave function there exists a~\emph{concave majorant} defined by
\[
\widetilde{\varphi}(t) = \inf_{s>0} \Big(1 + \frac{t}{s}\Big) \varphi(s)\,,
\]
which satisfies $\varphi(t) \leq \widetilde{\varphi}(t) \leq 2 \varphi(t)$ for all $t>0$.

For each given $\varphi\in \mathcal{Q}$, the \emph{Marcinkiewicz space} $M_\varphi$ on $\mathbb{R}_{+}$ is the r.i. space
of all $f\in L^0(\lambda)$ equipped with the norm
\[
\|f\|_{M_\varphi} := \sup_{t>0} \varphi(t)f^{**}(t)\,,
\]
where $f^{**}(t) := \frac{1}{t} \int_0^t f^*(s)\,ds$ for all $t>0$.

We will consider the Lorentz space $\Lambda_\varphi(\mathbb{R}^n)$ and the Marcinkiewicz space $M_{\varphi}(\mathbb{R}^n)$ over
the measure space $(\mathbb{R}^n, \lambda_n)$. We will use the K\"othe duality between Lorentz and Marcinkiewicz spaces,
which states that for any $\varphi \in \mathcal{P}$ with $\varphi(0+)=0$, we have
\[
\Lambda_\varphi(\mathbb{R}^n)' = M_{\varphi_{*}}(\mathbb{R}^n)
\]
with equality of norms. As a consequence, we have the following variant of H\"older's inequality (see, e.g., \cite[Theorem 5.2]{KPS} or
\cite[Chapter 1,  Theorem 2.4]{BSInterpol88}):
\begin{equation}\label{E:holder}
\int_{\mathbb{R}^n} |f g|\,d\lambda_n \leq \|f\|_{\Lambda_\varphi(\R^n)}\,\|g\|_{M_{\varphi_{*}}(\R^n)}\,.
\end{equation}
In what follows,  for simplicity of notation, we often   write
$\Lambda_\varphi$ and $M_\varphi$ for short instead of $\Lambda_{\varphi}(\mathbb{R}^n)$ and $M_{\varphi}(\mathbb{R}^n)$.

We will also consider a~class $\mathcal{B}$ of all measurable functions $\psi \colon \mathbb{R}_{+} \to \mathbb{R}_{+}$ such that the
function $m_{\psi}$ is finite and measurable, where
\[
m_{\psi}(t):= \sup_{s>0} \frac{\psi(st)}{\psi(s)}, \quad\, t>0\,.
\]
The lower and the upper index of a function $\psi\in \mathcal{B}$ are defined by
\[
\gamma_\psi= \lim_{t\to 0+} \frac{\log m_{\psi}(t)}{\log t}, \quad\, \delta_\psi = \lim_{t\to \infty} \frac{\log m_{\psi}(t)}{\log t}\,.
\]
We have $-\infty <\gamma_\psi \leq \delta_\psi <\infty$ (see \cite[Section 2, p.~53]{KPS}). Note that $\varphi, \psi \in \mathcal{B}$
with $\varphi \approx \psi$ implies $\gamma_\varphi = \gamma_\psi$ and $\delta_\varphi = \delta_\psi$.

\vspace{2 mm}

In the sequel we will use the following properties without any references:

\vspace{1.5 mm}

\noindent
\textup{(i)} Every function $\psi\in \mathcal{B}$ with $0<\gamma_\psi \leq \delta_\psi<1$ is equivalent to its \emph{concave majorant}
(see \cite[Corollary 2, p.~55]{KPS}).\\

\noindent
\textup{(ii)} If $\varphi\in \mathcal{P}$ with $\gamma_\varphi>0$, then it follows from \cite[Lemma 2.1.4]{KPS} that
\begin{equation}\label{*}
\varphi(t) \approx \int_0^t \frac{\varphi(s)}{s}\,ds, \quad\, t>0\,.
\end{equation}
In particular this implies that,
\[
\|f\|_{\Lambda_\varphi}  \approx \int_0^\infty f^{*}(t) \frac{\varphi(t)}{t}\,dt, \quad\, f\in \Lambda_{\varphi}(\mathbb{R}^n)\,,
\]
up to multiplicative constants depending only on $\varphi$.
\\

\noindent
\textup{(iii)} If $\varphi \in \mathcal{Q}$ with $\delta_\varphi <1$, then $\gamma_{\varphi_{*}} = 1 -\delta_\varphi>0$
and so by applying~\eqref{*}, we conclude that
\[
\|f\|_{M_\varphi} \approx \sup_{t>0} \varphi(t) f^{*}(t), \quad\, f\in M_\varphi(\mathbb{R}^n)\,.
\]

\noindent
\textup{(iv)} If $\varphi \in \mathcal{P}$ with $0<\gamma_\varphi \leq \delta_\varphi <1$, then the Lorentz space
$\Lambda_\varphi(\mathbb{R}^n)$ is separable (by $\varphi(0+)=0$ and $\varphi(+\infty)=\infty$). In particular,
it follows that the space $\mathscr C_0^\infty(\rn)$ is dense in
$\Lambda_\varphi(\mathbb{R}^n)$. \\

Throughout the paper we consider two families of special concave functions associated with indices $s\in (0,1)$ and
$\beta \in \R$. The function $\phi_{s,\beta}$ was defined in~\eqref{E:phi}. In addition, we let $\om_{s,\be}$ be
a~concave function satisfying
\[
\om_{s,\be}(t) \approx t^s \log^{\beta}(e+t), \quad t>0\,.
\]
Basic properties of the functions $\phi_{s,\beta}$ and $\om_{s,\be}$ are summarized in the following proposition.

\begin{Prop}
\label{phiomega}
Given $s>0$, $\beta \in \R$, consider the functions $\phi$ and $\omega$ defined by $\phi(t):=
t^s \log^\beta\big(e + \frac{1}{t}\big)$ and $\omega(t):= t^s \log^\beta(e + t)$ for all $t>0$. Then
$\phi, \omega \in \mathcal{B}$ with $\gamma_\phi = \delta_\phi = s$ and $\gamma_\omega = \delta_\omega = s$. If,
in addition, $s\in (0, 1)$, then $\phi$ and $\omega$ are equivalent to their concave majorant denoted by
$\phi_{s, \beta}$ and $\omega_{s, \beta}$, respectively.
\end{Prop}

\begin{proof}
We first focus on the case when $\beta\geq 0$. We observe that, for all $a, b \geq 0$ we have
\begin{align*}
\log (e + ab) & < \log [(e + a)(e+ b)] = \log(e+ a) + \log(e+b) \\
& \leq 2[\log(e+ a)]\,[\log(e+b)]\,.
\end{align*}
This shows that, for $C= 1/\log^{\beta}(e+1)$, we have
\[
C \phi(t) \leq m_\phi(t) = \sup_{r>0} \frac{\phi(rt)}{\phi(r)} \leq 2^{\beta} \phi(t), \quad\, t>0\,,
\]
and so $m_\phi \approx \phi$. Since $\phi$ is continuous, it follows that $\phi\in \mathcal{B}$ and we have
\begin{align*}
\gamma_\phi & = \lim_{t\to 0+} \frac{\log m_\phi(t)}{\log t} =  \lim_{t\to 0+} \frac{\log \phi(t)}{\log t}
= s + \beta \lim_{t\to 0+} \frac{\log (\log(e + t^{-1}))}{\log t} = s
\end{align*}
and
\[
\delta_\phi = s + \beta \lim_{t\to \infty} \frac{\log (\log(e + t^{-1}))}{\log t} = s\,.
\]
Similarly, we deduce that  $\omega \in \mathcal{B}$  with $m_\omega \approx \omega$ and
$\gamma_\omega = \delta_\omega = s$.

Notice that
$
m_{\phi_\beta} (t) = m_{\omega_{-\beta}}(t)
$
and  $m_{\omega_\beta}(t) = m_{\phi_{-\beta}}(t)$
for $t>0$ and $\beta\in \mathbb R$, where we set $\phi_\beta: = \phi$ and $\omega_\beta:= \omega$.
These equalities yield the   conclusion for $\beta<0$ using the  case  $-\beta>0$. If $s\in (0, 1)$,
then the required statements about concave majorants follow from the preceding results combined with property (i).
\end{proof}

We will need the following lemma. For completeness we include a proof.

\begin{Lemma}
\label{estimate**}
Let $\psi \in \mathcal{B}$ with $\delta_\psi <p$ for some $0<p\leq 1$. Then there exists a~constant $C=C(\psi, p)>0$
such that, for any $h\in L^0(\lambda_n)$, we have
\[
\int_0^\infty h^{**}(t)^{p}\,\frac{\psi(t)}{t}\,dt \leq C \int_0^\infty h^{*}(t)^p\,\frac{\psi(t)}{t}\,dt\,.
\]
\end{Lemma}

\begin{proof}
Observe that for any $t>0$, we have an obvious estimate
\[
h^{**}(t) = \int_0^1 h^{*}(ts)\,ds = \sum_{k=1}^{\infty} \int_{2^{-k}}^{2^{-k +1}} h^{*}(ts)\,ds\,
\leq \sum_{k=1}^{\infty} 2^{-k} h^{*}(t/2^k)\,.
\]
Combining with subadditivity of the function $ t\mapsto t^p$, defined on $ \mathbb{R}^{+} $, yields
\begin{align*}
\int_0^\infty h^{**}(t)^{p}\,\frac{\psi(t)}{t}\,dt
& \leq \sum_{k=1}^{\infty} 2^{-kp}\int_0^{\infty} h^{*}(t/2^k)^{p}\,\frac{\psi(t)}{t}\,dt \\
& = \sum_{k=1}^{\infty} 2^{-kp}\int_0^{\infty} h^{*}(t)^{p}\,\frac{\psi(2^{k}t)}{t}\,dt \\
& \leq C(p, \psi) \int_0^{\infty} h^{*}(t)^{p}\,\frac{\psi(t)}{t}\,dt\,,
\end{align*}
where $C(p, \psi):= \sum_{k=1}^{\infty} 2^{-kp}\,m_{\psi}(2^k)$.

We complete the proof by showing that $C(p, \psi)<\infty$. To see this observe that by $\delta_\psi <p$, we can
find $\varepsilon >0$ so that $\alpha:= p- \delta_\psi -\varepsilon>0$. It follows from the definition of
$\delta_\psi$ that there is an integer $k_0=k_0(\varepsilon)>0$ such that $m_\psi(2^k) \leq 2^{k(\delta_{\psi}
+ \varepsilon)}$ for each $k\geq k_0$
and hence
\[
\sum_{k= k_0}^{\infty} 2^{-kp}\,m_{\psi}(2^k) \leq \sum_{k=k_0}^{\infty} 2^{-k\alpha} <\infty\,.
\]
This concludes the proof.
\end{proof}

We now prove the following result on boundedness of the Fourier transform between corresponding Lorentz spaces
on $\mathbb{R}^n$. Before doing so, we point out that various variants of the Hausdorff-Young inequality in the setting of Lorentz spaces
are available in the literature (see, e.g.,~\cite{BH}, \cite{RS}, \cite{S}) but the result below appears to be new.

\begin{Lemma}
\label{lorentz}
Let $\varphi \in \mathcal{P}$ such that $1/2 <\gamma_\varphi \leq \delta_\varphi <1$. Then there exists a~constant $C>0$
such that, for any $f\in \Lambda_\varphi$, we have
\[
\|\widehat{f}\|_{\Lambda_{\psi}} \leq C \|f\|_{\Lambda_\varphi}\,,
\]
where $\psi(t):=t \varphi(1/t)$ for all $t>0$.
\end{Lemma}

\begin{proof}
Since $\varphi$ is concave, it is easy to check that $\psi$ is a~concave function
on $(0, \infty)$. Clearly, $m_{\psi}(t) = t\,m_\varphi(1/t)$ for all $t>0$ and so
\begin{align*}
\gamma_{\psi} = 1- \delta_\varphi \quad\, \text{and\,\, $\delta_{\psi} = 1- \gamma_\varphi$}\,.
\end{align*}
Hence it follows by assumption on indices of $\varphi$ that $0<\gamma_{\psi} \leq \delta_{\psi} <1/2$.

We use the pointwise estimate for the Fourier transform due to Jodeit and Torchinsky \cite[Theorem 4.6]{JT}, which
states that there exists a constant $D>0$ such that, for any $f\in L^1(\mathbb{R}^n) + L^2(\mathbb{R}^n)$, we have
\[
\int_0^t (\widehat{f}\,)^{*}(s)^2\,ds \leq D \int_0^t \bigg(\int_0^{\frac{1}{s}} f^{*}(\tau)\,d\tau\bigg)^2\,ds\,.
\]
Combining this estimate with
\[
(\widehat{f}\,)^{*}(t) \leq \frac{1}{t^{1/2}}\bigg(\int_0^t ((\wh{f})^{*}(s))^2\,ds \bigg)^{\frac{1}{2}}, \quad\ t>0\,,
\]
we obtain that for any simple function $f\in \Lambda_{\varphi}$,
\begin{align*}
\|\widehat{f}\|_{\Lambda_\psi} & \leq D \int_0^\infty \bigg(\frac{1}{t}
\int_0^t\bigg(\int_0^{\frac{1}{s}} f^{*}(\tau)\,d\tau\bigg)^2\,ds\bigg)^{\frac{1}{2}}\,\frac{\psi(t)}{t}\,dt
= \int_0^\infty g^{**}(t)^{\frac{1}{2}}\,\frac{\psi(t)}{t}\,dt\,,
\end{align*}
where $g$ is given  by
\[
g(t):= \bigg(\int_0^{\frac{1}{t}} f^{*}(\tau)\,d\tau\bigg)^{2}, \quad\, t>0\,.
\]
Clearly, $g$ is non-negative and nonincreasing and so it follows from Lemma \ref{estimate**} (by $\delta_\psi <1/2$) that
there exists $C>0$ such that
\[
\int_0^{\infty} g^{**}(t)^{\frac{1}{2}} \frac{\psi(t)}{t} \,dt \leq C\int_0^\infty g(t)^{\frac{1}{2}}\frac{\psi(t)}{t}\,dt\,.
\]
In consequence, we obtain
\begin{align*}
\|\widehat{f}\|_{\Lambda_{\psi}} & \lesssim \int_0^{\infty} \bigg(\int_0^{\frac{1}{t}} f^{*}(s)\,ds\bigg)\frac{\psi(t)}{t}\,dt
= \int_0^{\infty} f^{**}(t^{-1}) \frac{\psi(t)}{t^2}\,dt \\
& = \int_0^{\infty} f^{**}(t) \psi(t^{-1})\,dt = \int_0^{\infty} f^{**}(t)\frac{\varphi(t)}{t}\,dt \\
& \lesssim \int_0^{\infty} f^{*}(t)\frac{\varphi(t)}{t}\,dt\,,
\end{align*}
where the last estimate follows by Lemma \ref{estimate**} with $p=1$ (by $\delta_\varphi <1$).

Recall that $\gamma_\varphi >0$ implies that
\[
\|f\|_{\Lambda_\varphi} \approx \int_0^{\infty} f^{*}(t)\,\frac{\varphi(t)}{t}\,dt\,.
\]
Thus the required estimate follows by density of simple functions in the Lorentz space $\Lambda_\varphi$.
\end{proof}

\section{Preliminary results}

In this section we prove various auxiliary results that will be crucial in the proof of Theorem~\ref{1/2MainResult}.
We start with an estimate for an integral.

\begin{Lemma}\label{below}
If $k\ge 2$,
$\alpha_1, \dots, \alpha_k >1$, $r_1,\dots,r_k>0$ are such that $(\alpha_1-1)/r_1 \leq (\alpha_2-1)/r_2 \leq \dots
\leq (\alpha_k-1)/r_k$ and $a>1$, then
\begin{equation}
\label{induction}
\idotsint\limits_{\substack{u_{ 1} ,  \dots , u_{ k}\ge 1 \\  u_{ 1}^{r_{ 1}}\cdots u_{ k}^{r_{ k}} >a} }
u_{1}^{-\alpha_{1}}\cdots u_{ k}^{-\alpha_{ k}} \,du_{1}\cdots du_{k} \approx a^{\f {1-\alpha_1}{r_1}} [\log (e+ a )]^{d'}\,,
\end{equation}
up to multiplicative constants independent of $a$. Here, $d'$ is the number of elements in
$\{ (\alpha_2-1)/r_2,(\alpha_3-1)/r_3,\dots, (\alpha_k-1)/r_k \}$ that are equal to $(\alpha_1-1)/r_1$.
\end{Lemma}

\begin{proof}
To prove  \eqref{induction} we proceed by induction. First we verify the case $k=2$. In this case the $u_1$ integral
is over the region $u_1\ge \max\big\{ 1, (au_2^{-r_2})^{1/r_1}\big\}$ and so
evaluating the $u_1$ integral gives
\begin{equation}
\label{063gyui}
\iint\limits_{ \substack{u_{ 1} ,  u_{ 2}\ge 1 \\  u_{ 1}^{r_{ 1}}  u_{ 2}^{r_{ 2}} >a} }
u_{ 1}^{-\alpha_1}  u_{ 2}^{-\alpha_2}\,du_1 du_2
= C   \int_{u_2=1}^{\nf} u_2^{-\alpha_2}\,\max \big\{1, (au_2^{-r_2})^{1/r_1}\big\}^{1-\alpha_1}\,du_2 \,.
\end{equation}
If the maximum equals $1$, then the  integral is over the region $a^{1/r_2} \le u_2 <\nf$ and the $u_2$
integration produces $C a^{1/r_2(1-\alpha_2)} \leq C a^{1/r_1(1-\alpha_1)}$. Thus, the corresponding part
of~\eqref{063gyui} is bounded from above by the right-hand side of~\eqref{induction}, and if $a\in (1,2)$
then one has the lower bound as well.

If the maximum equals $(au_2^{-r_2})^{1/r_1}$, then the integral is over the region $1 \le u_2\le a^{1/r_2}$
and so the corresponding part of~\eqref{063gyui} becomes
\[
\int_{u_2=1}^{a^{\f 1{ r_2}}} u_2^{-\alpha_2} (au_2^{-r_2})^{(1 -\alpha_1)/r_1}\,du_2
= a^{\f {1-\alpha_1}{r_1} }    \int_{u_2=1}^{a^{\f 1{r_2}}} u_2^{-\alpha_2 - {r_2}/{r_1}(1-\alpha_1)}\,du_2\,.
\]
Now if $(\alpha_2-1)/r_2>(\alpha_1-1)/r_1$ then this is bounded from above by the right-hand side of~\eqref{induction}
with $d'=0$ and if $(\alpha_2-1)/r_2=(\alpha_1-1)/r_1$ then the same estimate holds with $d'=1$. In addition, one has
the corresponding lower bound if $a>2$. This concludes the proof of estimate  \eqref{induction} when $k=2$.

Assume by induction that \eqref{induction} holds for an integer $k-1$ (in place of $k$). Then
\[
\idotsint\limits_{ {\substack{ u_{1} ,  \dots , u_{ k}\ge 1 \\  u_{ 1}^{r_1} \cdots u_{ k}^{r_k }  > a}}}
u_{1}^{-\alpha_1}\cdots u_{k}^{-\alpha_k}\,du_{1}\cdots du_{k} = \int_{u_2 =1}^\nf \cdots \int_{u_k=1}^\nf
\frac{\int_{u_1= L}^\nf u_1^{-\alpha_1}\,du_1} {u_k^{\alpha_k}\cdots u_2^{\alpha_2}}\,\,d u_k\cdots  d u_2\,,
\]
where $L=\max\{1,  (a u_2^{-r_2} \cdots u_k^{-r_k})^{1/r_1}\}$. As $\alpha_1>1$, the $u_1$ integral is convergent
and the preceding expression equals
\begin{equation}
\label{gy350}
c \int_{u_2 =1}^\nf \cdots \int_{u_k=1}^\nf \max\Big\{1,  (a u_2^{-r_2} \cdots u_k^{-r_k})^{\f 1{r_1}}\Big\}^{1-\alpha_1}
\frac{d u_k}{u_k^{\alpha_k}} \cdots \frac{du_2}{u_2^{\alpha_2}}\,.
\end{equation}
The part of the  integral in \eqref{gy350} over the set where the maximum equals $1$ is
\begin{equation}\label{gy351}
c \idotsint\limits_{\substack{u_{2}, \dots , u_{k} \ge 1 \\  u_{ 2}^{r_{ 2}}\cdots u_{ k}^{r_{ k}} > a} }
u_k^{-\alpha_k}\cdots u_2^{-\alpha_2}\,du_{k} \cdots du_{2} \approx {a^{\f {1-\alpha_2}{r_2}}}
\log^{d{''}}(e + a)\,,
\end{equation}
where the equivalence holds by the induction hypothesis and $d{''}$ is the number of elements in
$\{(\alpha_3-1)/r_3,\dots, (\alpha_k-1)/r_k \}$ that are equal to  $ (\alpha_2-1)/r_2$. Note that if
$(\alpha_1-1)/r_1<(\alpha_2-1)/r_2$, then the expression on the right in  \eqref{gy351} is bounded from above by
\begin{equation}
\label{gy352}
C a^{\f {1-\alpha_1}{r_1}} \log^{d'}(e+ a)\,.
\end{equation}
Now if $(\alpha_1-1)/r_1=(\alpha_2-1)/r_2$, then we have $d'=d''+1$ and then the expression on the right in
\eqref{gy351} is also bounded by  \eqref{gy352}. In addition, we also have the corresponding lower
bound in both cases within the range $a\in (1,2)$. We now turn to the part of the integral  in \eqref{gy350}
over the set where the maximum equals $ (a u_2^{-r_2} \cdots u_k^{-r_k})^{\f 1{r_1}}$. It can be expressed as
\begin{equation}\label{E:second_term}
c a^{ \f {1-\alpha_1}{r_1} }   \idotsint\limits_{ \substack{u_{ 2} ,  \dots , u_{ k}\ge 1 \\
u_{ 2}^{r_{ 2}}\cdots u_{ k}^{r_{ k}} \le a} } u_k^{\f{ r_k}{r_1}(\alpha_1-1)
- \alpha_k} \cdots u_2^{\f{ r_2}{r_1}(\alpha_1-1)-\alpha_2}\,du_{k} \cdots du_{2}\,.
\end{equation}
First, we observe that we have the following upper bound for~\eqref{E:second_term}:
\begin{align*}
& c a^{ \f {1-\alpha_1}{r_1} }  \int_1^{a^{\frac{1}{r_2}}} \cdots \int_1^{a^{\frac{1}{r_k}}}
u_k^{\f{ r_k}{r_1} (\alpha_1-1)-\alpha_k}\cdots u_2^{\f{ r_2}{r_1}(\alpha_1-1)-\alpha_2}\,du_{k} \cdots  du_{2} \\
& \le C a^{\f {1-\alpha_1}{r_1}} \log^{d'}(e + a)\,,
\end{align*}
where the logarithm appears exactly when $ \f{ r_k}{r_1}=\f{\alpha_k-1}{\alpha_1-1}$ ($d'$ times) and the remaining
integrals produce a~constant. Conversely, if $a>2$ then we have an analogous lower bound for~\eqref{E:second_term} as well:
\begin{align*}
& c a^{ \f {1-\alpha_1}{r_1} }  \int_1^{a^{\frac{1}{(k-1)r_2}}} \cdots \int_1^{a^{\frac{1}{(k-1)r_k}}}
u_k^{\f{ r_k}{r_1} (\alpha_1-1)-\alpha_k}\cdots u_2^{\f{ r_2}{r_1}(\alpha_1-1)-\alpha_2}\,du_{k} \cdots  du_{2} \\
& \approx C a^{\f {1-\alpha_1}{r_1}} \log^{d'}(e + a) \,.
\end{align*}
The claim follows.
\end{proof}

We denote by $\mathcal{M}$ the strong maximal operator defined at point as the supremum of the averages of a given function
over all rectangles with sides parallel to the axes that contain the point. Then we define
\[
\mathcal{M}_{L^q} (g) (x_1,\dots , x_n) = \mathcal{M}(|g|^q) (x_1,\dots , x_n)^{\frac 1q}\,,
\]
a version of the strong maximal function with respect to an exponent $q\in (1,\infty)$.

\begin{Lemma}
\label{lemmas1=s2}
Let  $0<1/q<s_1\le s_2\le \cdots\le s_n<1$. Suppose that exactly $d$ of the numbers $s_2,\dots, s_n$ are equal to $s_1$,
where $1\le d\le n-1$. Then for  $g$ in $L^1_{loc} (\rn)$ with $\mathcal{M}_{ L^q} (g)(0) = 1 $ and $a>0$ we have
\begin{equation}\label{abcabc}
\left\lvert \left\{ y\in \rn\setminus [-1,1]^n : \,\, \dfrac{ \lvert g (y )\rvert}{ \prod_{i=1}^n
(1+\lvert y_i\rvert)^{s_i}} >a \right\} \right\rvert \le C a^{-\f1{s_1}} \log^{d}\Big(e+\f 1 a\Big)\,.
\end{equation}	
\end{Lemma}

\begin{proof}
For $j_1, \dots , j_n$ nonnegative integers define
\[
R_{j_1,\dots , j_n} =\Bigg\{(y_1,\dots , y_n)\in \rn: \,\,
\begin{cases} 2^{j_i}< |y_i|\le 2^{j_i+1} &\text{if $j_i\ge 1$} \\
|y_i|\le 1 &\text{if $j_i=0$} ,
\end{cases}
\q 1\leq i \leq n \Bigg\}
\]
and notice that the family of rectangles $R_{j_1,\dots , j_n}$ is a tiling of $\rn$ when $j_1, \dots , j_n$ run over all
nonnegative integers.
	
In the sequel we denote by $y$ the vector $(y_1,\dots, y_n)$. For $a>0$ and $j_1,\dots, j_n$ nonnegative integers, we have
\begin{equation*}
\lvert \{y \in R_{j_1,\dots ,j_n}:\,\, \lvert g(y ) \rvert >a\} \rvert  \leq \dfrac{1}{a^q} \int_{R_{j_1,\dots, j_n}} |g(y )|^q dy
\leq a^{-q}\,2^{j_1+\cdots +j_n+2n}
\end{equation*}
since we are assuming that $\mathcal{M}_{ L^q} (g)(0) = 1 $. Thus, in view of the trivial estimate $|R_{j_1,\dots , j_n}|
\le 2^{j_1+\cdots + j_n+2n}$, we obtain
\begin{equation}
\label{kdkkkd3}
\lvert \{y \in R_{j_1,\dots, j_n}:\,\, \lvert g(y ) \rvert >a\} \rvert \leq
2^{2n} 2^{j_1+\cdots +j_n } \min \big\{1, a^{-q}\big\}\,.
\end{equation}
It follows from \eqref{kdkkkd3} that, for all $j_1,\dots, j_n\ge 0$, we have
\begin{align}
\begin{split}\label{Doublest}
&\left\lvert \left\{ y\in R_{j_1,\dots, j_n} :\,\, \dfrac{ \lvert g(y )\rvert}{(1+\lvert y_1\rvert)^{s_1}\cdots
(1+\lvert y_n\rvert)^{s_n}} >a \right\} \right\rvert  \\
&\qquad\qquad\qquad\qquad\le \,\, 2^{2n} 2^{j_1+\cdots +j_n} \min\big\{1, (a 2^{ j_1s_1+ \cdots j_ns_n})^{-q}\big)\,.
\end{split}
\end{align}

We let $g_1= g\chi_{\rn\setminus R_{0,\dots , 0}}$. Using  \eqref{Doublest}, we get that
\begin{align}		
&\left\lvert \left\{ y\in \rn : \,\, \dfrac{ \lvert g_1(y )\rvert}{(1+\lvert y_1\rvert)^{s_1}\cdots (1+\lvert y_n\rvert)^{s_n}} >a \right\} \right\rvert  \notag \\
	& \le   \quad   \sum_{\substack{ j_1,\dots, j_n=0\\j_1+\cdots +j_n>0}}^{\infty}  \left\lvert \left\{ y\in R_{j_1,\dots , j_n} : \,\, \dfrac{ \lvert g(y )\rvert}{(1+\lvert y_1\rvert)^{s_1}\cdots (1+\lvert y_n\rvert)^{s_n}} >a \right\} \right\rvert  \notag\\
	& \le   \sum_{\substack{ j_1,\dots ,j_n=0\\j_1+\cdots+ j_n>0}}^{\infty}  2^{j_1+\cdots + j_n+2n}
\min\big\{1, a^{-q} (2^{j_1s_1+\cdots +j_ns_n})^{-q}\big\} \notag \\
	& \le 2^{2n+s_1q+\cdots+s_nq}\idotsint\limits_{[0,\nf)^n\setminus [0,1]^n}
	\min\!\big\{1, a^{-q} \prod_{\rho=1}^n \max\{1,t_\rho^{s_\rho}\}^{-q}\big\}\,dt_1\cdots dt_n \notag \\
	& =: 2^{2n+s_1q+\cdots+s_nq} \, I (a,n), \label{Ian}
	\end{align}
where the last inequality follows from the monotonicity of the integrand. Let $S$ be the set of all
$(t_1, \dots t_n)\in [0,\nf)^n\setminus [0,1]^n$ such that
\begin{equation}\label{est1010}
a \max\{1, t_1^{s_1}\} \cdots \max\{1, t_n^{s_n}\} \leq 1\,.
\end{equation}
If $S$ is nonempty, then we must have $a\le 1$.
Let us fix a two-set partition   $I=\{i_1, \dots , i_m\}$ and  $J=\{j_1,\dots , j_k\}$ of $\{1, 2, \dots , n\}$.
We split $S$ as a union of sets $S_{I,J}$ (ranging over all such pairs of partitions) for which
\begin{equation}\label{est1011gggg}
(t_1, \dots, t_n) \in S_{I,J} \iff  \textup{$t_i\le 1$ for all $i\in I$ and  $ t_j>1$ for all $j\in J$.}
\end{equation}
Then the $n$-dimensional measure  $|S_{I,J}|$ of $S_{I,J}$ is at most the $k$-th dimensional measure of
\[
S_{I,J}^k=\Big\{(t_{j_1}, \dots , t_{j_k}):\,\, t_{j_1}^{s_{j_1}}\cdots t_{j_k}^{s_{j_k}} \le \f1a\Big\} \cap [1,\nf)^{k}\,,
\]
as the vector of the remaining $m$ coordinates is contained in the cube $[0,1]^m$ which has $m$-th dimensional measure equal to $1$.
Let us assume, without loss of generality, that
$s_{j_1}\le s_{j_2}\le  \dots \le  s_{j_k} $ (i.e., $j_1<j_2<\cdots<j_k$).

We make the following observation: if $(t_{j_1}, \dots  , t_{j_k}) \in S_{I,J}^k$, then
\[
1\le t_{j_i} \le a^{-\f 1{s_1}}, \quad\, 1\le i\le k\,.
\]
Indeed, as all $t_{j_i} \ge 1$, we have
$1\le  t_{j_i}^{s_{j_i}} \le t_{j_1}^{s_{j_1}}\cdots t_{j_k}^{s_{j_k}} \le a^{-1}$, which  implies that
$1\le t_{j_i}\le a^{-1/s_{j_i}} \le a^{-1/s_{1}}$. Thus we conclude that
 \begin{align*}
| S_{I,J}^k | \le &\, \int_{t_{j_2}=1}^{a^{- \f1{s_1}}} \cdots \int_{t_{j_k}=1}^{a^{- \f1{s_1}}} \Big| \Big\{  t_{j_1} :  \,\,
1 \le  t_{j_1} \le a^{-\f1{s_{j_1}}}
t_{j_2}^{-\f { s_{j_2}}{s_{j_1}} }\cdots t_{j_k}^{-\f{ s_{j_k}}{ s_{j_1}}} \Big\} \Big|
\, d t_{j_k}\cdots  dt_{j_2} \\
\le & \,\, a^{-\f1{s_{j_1}}} \chi_{a\le 1}\int_{t_{j_2}=1}^{a^{- \f1{s_1}}}\cdots \int_{t_{j_k}=1}^{a^{- \f1{s_1}}}
t_{j_2}^{-\f { s_{j_2}}{s_{j_1}} }\cdots t_{j_k}^{-\f{ s_{j_k}}{ s_{j_1}}}
\, d t_{j_k}\cdots  dt_{j_2} \\
\le & \,\, C a^{-\f1{s_{j_1}}}\chi_{a\le 1} \log^{d'}\Big( \f1 {a} \Big)^{\f 1{s_1}}  \le
C' a^{-\f1{s_{ 1}}}\chi_{a\le 1} \log^{d} \Big(e+ \f1 {a} \Big)\,,
\end{align*}
where $d'$ is the number of elements of the set    $   \{ s_{j_2}  , \dots ,    s_{j_k}\} $ that are equal to $s_{j_1}$.
The integrals associated with these variables produce a logarithm, while all other integrals are convergent   on $[1,\nf)$.
The last inequality holds independently of the relationship between $d$ and $d'$ if $s_{j_1}>s_1$, while if $s_{j_1}=s_1$
then it is satisfied since $d'\leq d$. Summing over all partitions $(I,J)$ of $\{1,2,\dots , n\}$ yields the required estimate
for $I(a,n)$, defined in \eqref{Ian}, whenever \eqref{est1010} holds.

Now let $S'$ be the set of all $(t_1, \dots t_n)\in [0,\nf)^n\setminus [0,1]^n$ such that
\begin{equation}\label{est1011}
a \max\{1,t_1^{s_1}\}\cdots  \max\{1, t_n^{s_n}\} > 1 \,.
\end{equation}
Then $S'$ is complementary to $S$ in $[0,\nf)^n\setminus [0,1]^n$.
Writing $S'$ as a union of sets $S'_{I,J}$ over all partitions $(I,J)$ of $\{1,2,\dots, n\}$
as in \eqref{est1011gggg}, matters reduce to estimating the  integral
\begin{equation}\label{ioiejjd}
\f{1}{a^q} \idotsint\limits_{ \substack{t_{j_1}, \dots, t_{j_k} \ge 1 \\
t_{j_1}^{s_{j_1}}\cdots t_{j_k}^{s_{j_k}} >\f1a}}
(t_{j_1}^{s_{j_1}}\cdots t_{j_k}^{s_{j_k}})^{-q}\,dt_{j_1} \cdots dt_{j_k}
\end{equation}
for each subset $\{j_1, \dots , j_k\}$ of $\{1,\dots ,n\}$. Now if $a>1$, the integral in \eqref{ioiejjd} is over the set
$[1,\nf)^k$ and, as $s_{j_1}q>1, \dots , s_{j_k}q>1$, the expression
in \eqref{ioiejjd}  is bounded by
\[
C \, a^{-q} \chi_{a>1}  \le C \, a^{-\f1{s_1}}\,,
\]
since $q>1/s_1$. So we focus attention to the case $a\le 1$ in \eqref{ioiejjd}. Let us again assume, without loss of generality,
that $s_{j_1}\le s_{j_2} \le \dots \le s_{j_k}$. To estimate  \eqref{ioiejjd} we use Lemma~\ref{below}. Inequality \eqref{induction}
in this lemma implies that, if $d'$ is the number of terms in $\{s_{j_2}, \dots, s_{j_k}\}$ that are equal to $s_{j_1}$, then
\eqref{ioiejjd} is bounded by
\[
C a^{-\f 1{s_{j_1}}} \log^{d'}\big(e+ a^{-1}\big) \le C a^{-\f 1{s_{1}}}
\log^{d}\big(e+ a^{-1}\big)\,,
\]
where the last inequality is due to the fact that either $s_{j_1}=s_1$ and $d'\le d$, or $s_{j_1} <s_1$ (as $a<1$).

Summing over all partitions $(I,J)$ of $\{1,2,\dots , n\}$ yields the required estimate for $I(a,n)$, defined in
\eqref{Ian}, whenever \eqref{est1011} holds. This completes the proof of \eqref{abcabc}.
\end{proof}

\begin{Corollary}\label{cor87}
Let $0<1/q<s_1 \le s_2 \le \cdots \le s_n<1$ with exactly $d$
numbers among $s_2,\dots,s_n$ being equal to $s_1$.
For $g$ in $L^1_{loc} (\rn)$ with $\mathcal{M}_{ L^q} (g)(0) = 1 $ we have
\bee\lab{1stcon}
\bigg( \dfrac{ \lvert g (y_1,\dots, y_n )\rvert\chi_{\rn\setminus [-1,1]^n} }{ \prod_{i=1}^n  (1+\lvert y_i\rvert)^{s_i}}\bigg)^*(t)
\le \f{C}{\om_{s_1,-s_1d}(t)}\,.
\eee

\end{Corollary}

\begin{proof}
Note that the inverse function to
\[
(0, \infty) \ni a\mapsto a^{-\f1{s_1} } \log^{d}\Big(e+\f 1 a\Big)\,,
\]
that appears in Lemma~\ref{lemmas1=s2}, is equivalent to
\[
t \mapsto t^{- s_1} \log^{s_1d}(e+t ) =   \big(\om_{s_1, -s_1d}(t)\big)^{-1}\,.
\]
This proves \eqref{1stcon}.
\end{proof}

\begin{Prop}\label{weakorliczcube}
Assume that $h$ is supported in the cube $[-1,1]^n=Q_0$ and that $s_1>1/r>1/q>0$. Then
\begin{equation}\label{177}
\Norm{h}{M_{\om_{s_1,-s_1d}}} \lesssim \Norm{h}{L^{r,\infty}} \lesssim  \mathcal{M}_{ L^q} (h)(0)\,.
\end{equation}
\end{Prop}

\begin{proof}
We first notice that the function $h$ is supported in a set of measure $2^n$, and therefore $h^*(t)=0$ if $t>2^n$.
Since the function $\om_{s_1,-s_1d}(t)/t$ is non-increasing, we have
\begin{align*}
\|h\|_{M_{\om_{s_1,-s_1d}}}
&=\sup_{t>0} \frac{\om_{s_1,-s_1d}(t)}{t} \int_0^t h^*(s)\,ds
=\sup_{t\in (0,2^n)} \frac{\om_{s_1,-s_1d}(t)}{t} \int_0^t h^*(s)\,ds\\
&\lesssim \sup_{t\in (0,2^n)} \frac{t^{\frac{1}{r}}}{t} \int_0^t h^*(s)\,ds
\lesssim \|h\|_{L^{r,\infty}}\,.
\end{align*}
Notice that the first inequality above makes use of the fact that
$\om_{s_1,-s_1d}(t) \lesssim t^{1/r}$ for $t\in (0,2^n)$ as $1/r<s_1$. This proves the first inequality
in \eqref{177}. The second inequality in \eqref{177} follows from the natural embedding of $L^q([-1,1]^n)$
in $L^{r,\infty} ([-1,1]^n) $, as $r<q$.
\end{proof}

Combining the results of Corollary~\ref{cor87} and Proposition~\ref{weakorliczcube} we obtain the following.

\begin{Corollary}\label{cor88}
Let $0<1/q<s_1 \le s_2 \le \cdots \le s_n<1$ with exactly $d$
numbers among $s_2,\dots,s_n$ being equal to $s_1$.
For $g$ in $L^1_{loc} (\rn)$ with $\mathcal{M}_{ L^q} (g)(0) = 1 $ we have
\bee\lab{1stcon11}
\bigg( \dfrac{ g (y_1,\dots , y_n ) }{ \prod_{i=1}^n  (1+\lvert y_i\rvert)^{s_i}}\bigg)^*(t)
\le \f{C}{     \om_{s_1,-s_1d}(t)}\,.
\eee
Consequently, for any $g\in L^1_{loc}(\rn) $ and any $x=( x_1,\dots , x_n) \in \rn$, we have
\bee\lab{2ndcon}
\bigg\| \dfrac{g (x_1+2^{-j_1} y_1,
\dots, x_n+2^{-j_n} y_n)}{ \prod_{i=1}^n  (1+\lvert y_i\rvert)^{s_i}}\bigg\|_{M_{\om _{s_1,-s_1d}} (dy_1\cdots dy_n)}
\le C \mathcal M_{L^q}(g)(x),
\eee
for any $j_1,\dots, j_n\in \mathbb Z$.
\end{Corollary}

\begin{proof}
To  prove \eqref{1stcon11} we split $g= g_0+g_1$, where $g_0=g\chi_{[-1,1]^n} $ and
$g_1= g\chi_{\rn\setminus [-1,1]^n}$ and we apply Corollary~\ref{cor87} to $g_1$ and
 Proposition~\ref{weakorliczcube} to $g_0$. Now \eqref{1stcon11} applied to $g/ \mathcal{M}_{ L^q} (g)(0)$
 yields   \eqref{2ndcon}   when $x=0$ and $j_1=\cdots = j_n=0$.  The general
case of  \eqref{2ndcon} can be obtained by a translation and a dilation.
\end{proof}

\section{A limiting case Sobolev embedding}\label{S:embedding}

The following embedding of a Lorentz-Sobolev space into the space of essentially bounded functions is an
important ingredient for the proof of Theorem~\ref{1/2MainResult}.

\begin{Prop}\label{P:embedding}
Let $0<s_1 \leq s_2\leq \dots \leq s_n <1$, where exactly $d$ of the numbers $s_2,\dots,s_n$ are equal to $s_1$. Then
\[
\|f\|_{L^\infty(\mathbb R^n)} \lesssim \|\Gamma(s_1,\dots,s_n)f\|_{\Lambda_{\phi_{s_1,(1-s_1)d}}(\mathbb R^n)}.
\]
\end{Prop}

\begin{proof}
For a given function $f$, we denote
\[
g=\Gamma(s_1,\dots,s_n) f\,.
\]
We write $f=\Gamma(-s_1,\dots,-s_n) g = (G_{s_1} \otimes \dots \otimes G_{s_n})*g$, where $G_s$ is the
one-dimensional kernel of $(I-\partial^2)^{-s/2}$. We recall the estimates
\[
G_s(x)\lesssim |x|^{s-1} \quad \text{as } x\to 0
\]
\[
G_s(x)\lesssim e^{-c|x|} \quad \text{as } |x|\to \infty\,.
\]
The H\"older inequality~\eqref{E:holder} yields
\begin{align*}
\|f\|_{L^\infty(\R^n)}
&=\|(G_{s_1} \otimes \dots \otimes G_{s_n}) * g\|_{L^\infty(\R^n)}\\
&\leq \sup_{(\tilde{x}_1,\dots,\tilde{x}_n)\in \R^n} \int_{\R^n} G_{s_1}(x_1) \cdots G_{s_n}(x_n) |g(\tilde{x}_1-x_1,\dots,\tilde{x}_n-x_n)|\,dx_1\cdots dx_n\\
&\leq \|G_{s_1} \otimes \cdots \otimes G_{s_n}\|_{M_{\phi_{1-s_1,d (s_1-1)}}(\R^n)} \|g\|_{\Lambda_{\phi_{s_1,(1-s_1)d}}(\R^n)}\,,
\end{align*}
as for $t>0$ we have
\[
\f{t}{\phi_{s_1,(1-s_1)d}(t)}=t^{1-s_1}\log^{d(s_1-1)}(e+\tfrac 1t) = \phi_{1-s_1, d(s_1-1)}(t) .
\]
It remains to verify that
\begin{equation}\label{E:G}
G_{s_1} \otimes \dots \otimes G_{s_n} \in M_{\phi_{1-s_1, d(s_1-1)}}(\R^n)\,.
\end{equation}
Given a subset $I$ of $\{1,2,\dots,n\}$, we set $J=\{1,2,\dots,n\}\setminus I$ and write $I=(i_1,\dots,i_k)$
and $J=(j_1,\dots,j_{n-k})$ (there is a slight abuse of notation as one of the sets may be empty). Now observe that
\begin{align*}
&|\{(x_1,\dots,x_n)\in \R^n: ~ G_{s_1}(x_1)\dots G_{s_n}(x_n) >\lambda\}|\\
&\leq \sum_{I\subseteq \{1,2,\dots,n\}} |\{(x_{i_1},\dots,x_{i_k})\in (-1,1)^k, ~(x_{j_1},\dots,x_{j_{n-k}})\in (\R\setminus (-1,1))^{n-k}:\\
&\hspace{2in} |x_{i_1}|^{s_{i_1}-1} \cdots |x_{i_k}|^{s_{i_k}-1} e^{-c(|x_{j_1}|+\dots+|x_{j_{n-k}}|)}>\lambda\}|\\
&\lesssim \sum_{I\subseteq \{1,2,\dots,n\}} |\{(x_{i_1},\dots,x_{i_k})\in (0,1)^k, ~(x_{j_1},\dots,x_{j_{n-k}})\in (1,\infty)^{n-k}:\\
&\hspace{2in} x_{i_1}^{s_{i_1}-1} \cdots x_{i_k}^{s_{i_k}-1} e^{-c(x_{j_1}+\dots+x_{j_{n-k}})}>\lambda\}|\,.
\end{align*}

We denote
\begin{align*}
&S_{I,\lambda}=
\{(x_{i_1},\dots,x_{i_k})\in (0,1)^k, ~(x_{j_1},\dots,x_{j_{n-k}})\in (1,\infty)^{n-k}:\\
&\hspace{2in} x_{i_1}^{s_{i_1}-1} \cdots x_{i_k}^{s_{i_k}-1} e^{-c(x_{j_1}+\dots+x_{j_{n-k}})}>\lambda\}.
\end{align*}
We want to estimate $|S_{I,\lambda}|$. To this end, we fix $I\subseteq \{1,2,\dots,n\}$ and
$(x_{j_1},\dots,x_{j_{n-k}})\in (1,\infty)^{n-k}$. Further, for a fixed $\lambda>0$ we set
$a=\lambda e^{c(x_{j_1}+\dots+x_{j_{n-k}})}$. If $a>1$ then we estimate
\begin{align}
\label{E:a}
&|\{(x_{i_1},\dots,x_{i_k})\in (0,1)^k:~x_{i_1}^{s_{i_1}-1} \cdots x_{i_k}^{s_{i_k}-1} >a\}|\\
\nonumber
&=\idotsint\limits_{ \substack{x_{i_1} ,  \dots , x_{i_k}\in (0,1) \\
x_{i_1}^{s_{i_1}-1}\cdots x_{i_k}^{s_{i_k}-1}>a}} dx_{i_1}\cdots dx_{i_k}
=\idotsint\limits_{ \substack{u_{1} ,  \dots , u_{k}>1 \\  u_{1}^{1-s_{i_1}}\cdots u_{k}^{1-s_{i_k}}>a}}
u_1^{-2}\dots u_k^{-2}\,du_{1}\cdots du_k\\
\nonumber
&\lesssim a^{-\frac{1}{1-s_{i_1}}}\log^{d'}(e+a)\,,
\end{align}
where $d'$ is the number of elements from the set $\{s_{i_2},\dots,s_{i_k}\}$ that are equal to $s_{i_1}$.
We recall that the last inequality follows from Lemma~\ref{below}. Notice that the estimate~\eqref{E:a}
is true also if $a\leq 1$ as the measure of the set on the left-hand side is at most $1$, which is trivially
bounded by the right-hand side. We also observe that
\begin{align*}
a^{-\frac{1}{1-s_{i_1}}}\log^{d'}(e+a)
&\lesssim \lambda^{-\frac{1}{1-s_{i_1}}}
\log^{d'}(e+\lambda) e^{-\frac{c}{1-s_{i_1}}(x_{j_1}+\dots+x_{j_{n-k}})} (x_{j_1}+\dots+x_{j_{n-k}})^{d'}\\
&\lesssim \lambda^{-\frac{1}{1-s_{i_1}}}\log^{d'}(e+\lambda) e^{-c'(x_{j_1}+\dots+x_{j_{n-k}})}\,,
\end{align*}
where $c'<\frac{c}{1-s_{i_1}}$.

Thus, if $\lambda>1$ then we have
\begin{align*}
|S_{I,\lambda}| &\lesssim \idotsint\limits_{x_{j_1},\dots, x_{j_{n-k}}>1} \lambda^{-\frac{1}{1-s_{i_1}}} \log^{d'}(e+\lambda) e^{-c'(x_{j_1}+\dots+x_{j_{n-k}})}\,dx_{j_1}\dots dx_{j_{n-k}}\\
&\lesssim \lambda^{-\frac{1}{1-s_{i_1}}} \log^{d'}(e+\lambda)
\lesssim \lambda^{-\frac{1}{1-s_{1}}} \log^{d}(e+\lambda)\,.
\end{align*}
On the other hand, if $\lambda\leq 1$ then
\begin{align}\label{E:minimum}
&|\{(x_{i_1},\dots,x_{i_k})\in (0,1)^k: ~x_{i_1}^{s_{i_1}-1} \dots x_{i_k}^{s_{i_k}-1} e^{-c(x_{j_1}+\dots+x_{j_{n-k}})}>\lambda\}|\\
\nonumber
&\lesssim \min\{1,\lambda^{-\frac{1}{1-s_{i_1}}}\log^{d'}(e+\lambda) e^{-c'(x_{j_1}+\dots+x_{j_{n-k}})}\}\,.
\end{align}
If the minimum is equal to $1$ then $e^{c'(x_{j_1}+\dots+x_{j_{n-k}})}\leq \lambda^{-\frac{1}{1-s_{i_1}}} \log^{d'}(e+\lambda)$,
and so $x_{j_1}+\dots+x_{j_{n-k}} \lesssim \log(e\lambda^{-1})$. Then the measure of the corresponding part
of the set $S_{I,\lambda}$ is bounded by constant times
\begin{equation*}
\log^{n-k}(e\lambda^{-1}) \lesssim \lambda^{-\frac{1}{1-s_1}}\log^{d}(e+\lambda)\,.
\end{equation*}
Finally, if the minimum in~\eqref{E:minimum} is equal to
$\lambda^{-\frac{1}{1-s_{i_1}}} \log^{d'}(e+\lambda) e^{-c'(x_{j_1}+\dots+x_{j_{n-k}})}$, then
\[
x_{j_1}+\dots+x_{j_{n-k}} \geq \frac{1}{c'} \log(\lambda^{-\frac{1}{1-s_{i_1}}} \log^{d'}(e+\lambda),
\]
and the measure of the corresponding part of the set $S_{I,\lambda}$ is bounded by constant times
\begin{align*}
&\idotsint\limits_{\substack{x_{j_1},\dots,x_{j_{n-k}}>1\\x_{j_1}+\dots+x_{j_{n-k}} \geq \frac{1}{c'} \log(\lambda^{-\frac{1}{1-s_{i_1}}} \log^{d'}(e+\lambda))}} \!\!\!\!\!\!\!\!\!\!\!\!
\lambda^{-\frac{1}{1-s_{i_1}}} \log^{d'}(e+\lambda) e^{-c'(x_{j_1}+\dots+x_{j_{n-k}})}\,dx_{j_1}\dots dx_{j_{n-k}}\\
&\lesssim \int_{\frac{1}{c'} \log(\lambda^{-\frac{1}{1-s_{i_1}}}\log^{d'}(e+\lambda))}^\infty \lambda^{-\frac{1}{1-s_{i_1}}}
\log^{d'}(e+\lambda) e^{-c{'}r} r^{n-k-1}\,dr\\
&\lesssim \lambda^{-\frac{1}{1-s_{i_1}}}\log^{d'}(e+\lambda) \int_{\frac{1}{c'} \log(\lambda^{-\frac{1}{1-s_{i_1}}} \log^{d'}(e+\lambda))}^\infty e^{-c{''}r} \,dr\\
&\lesssim \lambda^{-(1-\frac{c{''}}{c'})\frac{1}{1-s_{i_1}}} \log^{(1-\frac{c{''}}{c'})d'}(e+\lambda)
\lesssim \lambda^{-\frac{1}{1-s_1}}\log^{d}(e+\lambda)\,.
\end{align*}
The last inequality holds since $c{''}$ can be chosen to be any number less than $c'$, and thus $1-c{''}/c'$ can be arbitrarily small.

Altogether, we proved
\[
|\{(x_1,\dots,x_n)\in \R^n:~G_{s_1}(x_1) \cdots G_{s_n}(x_n)>\lambda\}| \lesssim \lambda^{-\frac{1}{1-s_1}} \log^{d}(e+\lambda)\,.
\]
This yields
\[
(G_{s_1} \otimes \dots \otimes G_{s_n})^{*}(t) \lesssim t^{s_1-1}(\log(e+1/t))^{d(1-s_1)}, \quad t>0\,,
\]
which in turn implies~\eqref{E:G}.
\end{proof}

Next we show that the previous result is sharp, in the sense that the space $\Lambda_{\phi_{s_1,(1-s_1)d}}$ is locally the largest 
rearrangement-invariant space for which Proposition~\ref{P:embedding} holds.

\begin{Prop}\label{P:optimality}
Let $0<s_1 \leq s_2\leq \dots \leq s_n <1$, where exactly $d$ of the numbers $s_2,\dots,s_n$ are equal to $s_1$. Assume that $E$ 
is a~rearrangement-invariant space such that
\begin{equation}\label{E}
\|f\|_{L^\infty(\mathbb R^n)} \lesssim \|\Gamma(s_1,\dots,s_n)f\|_{E(\mathbb R^n)}.
\end{equation}
Then $E(\Omega) \hookrightarrow \Lambda_{\phi_{s_1,(1-s_1)d}}(\Omega)$ for all sets $\Omega \subseteq \R^n$ of finite measure.
\end{Prop}

\begin{proof}
To prove this claim, we set $g=\Gamma(s_1,\dots,s_n)f$ and rewrite inequality~\eqref{E} as
\begin{equation}\label{E:rewritten}
\|(G_{s_1} \otimes \cdots \otimes G_{s_n})* g\|_{L^\infty(\R^n)} \lesssim \|g\|_{E(\R^n)}\,,
\end{equation}
where $G_{s}$ is the one-dimensional kernel of $(I-\partial^2)^{-s/2}$. For a~given
$(\tilde{x}_1,\dots,\tilde{x}_n)\in \R^n$, we have
\begin{align}\label{E:duality}
&\sup_{\|g\|_{E(\R^n)}\leq 1} |(G_{s_1} \otimes \cdots \otimes G_{s_n})* g (\tilde{x}_1,\dots,\tilde{x}_n)|\\
\nonumber
&=\sup_{\|g\|_{E(\R^n)}\leq 1} \left|\int_{\R^n} G_{s_1}(x_1)\cdots G_{s_n}(x_n) g(\tilde{x}_1-x_1\,,
\dots,\tilde{x}_n-x_n)\,dx_1 \dots dx_n\right|\\
\nonumber
&=\|G_{s_1}\otimes \cdots \otimes G_{s_n}\|_{E'(\R^n)}\,,
\end{align}
where $E'$ is the K\"othe dual space of $E$.
Thus, \eqref{E:rewritten} implies that the function $G_{s_1}\otimes \cdots \otimes G_{s_n}$ belongs to $E'(\R^n)$.
We next find a lower bound for the distribution function of $G_{s_1}\otimes \dots \otimes G_{s_n}$. Since $G_{s_i}(x_i)
\approx x_i^{s_i-1}$ if $0<x_i<1$ and $i \in \{1,2,\dots,n\}$, we obtain for $\lambda >1$,
\begin{align*}
&|\{(x_1,\dots,x_n) \in \R^n: ~G_{s_1}(x_1)\cdots G_{s_n}(x_n) >\lambda\}|\\
&\gtrsim |\{(x_1,\dots,x_n) \in (0,1)^n: x_1^{s_1-1}\cdots x_n^{s_n-1}>\lambda\}|\\
&\approx \lambda^{-\frac{1}{1-s_1}} \log^d(e+\lambda)\,,
\end{align*}
where $d$ is the number of elements from the set $\{s_2,\dots,s_n\}$ that are equal to $s_1$. Note that the last equivalence 
follows by the calculation in~\eqref{E:a} and by Lemma~\ref{below}. This shows that
\[
(G_{s_1}\otimes \cdots \otimes G_{s_n})^*(t) \gtrsim t^{s_1-1} \log^{(1-s_1)d}\Big(e+\frac{1}{t}\Big), \quad t\in (0,t_0)
\]
for some $t_0>0$.  This shows that if
\[
g(t):= t^{s_1-1} \log^{(1-s_1)d}\Big(e+\frac{1}{t}\Big), \quad\, t>0\, , 
\]
then the function $g\chi_{(0, t_0)} \in E':= E'(0, \infty)$. 
To reach the conclusion we observe that the embedding $E(\Omega) \hookrightarrow \Lambda_{\phi_{s_1,(1-s_1)d}}(\Omega)$ is, 
by duality, equivalent to $M_{\phi_{1-s_1,(s_1-1)d}}(\Omega) \hookrightarrow E'(\Omega)$. Now, if $f$ is a~function satisfying 
$\|f\|_{M_{\phi_{1-s_1,(s_1-1)d}}(\Omega)} \leq 1$, then
\[
f^*(t) \leq f^{**}(t) \lesssim g(t), \quad\, t\in (0,|\Omega|)\,.
\]
Hence 
\begin{align*}
\|f\|_{E'(\Omega)} & = \|f^*\chi_{(0,|\Omega|)}\|_{E'} \lesssim \|g\chi_{(0,|\Omega|)}\|_{E'}\\
&\lesssim \|g\chi_{(0,t_0)}\|_{E'} + \|\chi_{(t_0,|\Omega|)}\|_{E'} \\
& \lesssim \|g\chi_{(0,t_0)}\|_{E'} \leq C\,.
\end{align*}

In this chain of inequalities we used the monotonicity and rearrangement-invariance of the norm in the space $E'$, the fact that 
the interval $(t_0,|\Omega|)$ can be split into a finite number of intervals of length at most $t_0$ and that the 
constant function on the interval $(0,t_0)$ is bounded from above by a multiple of the function 
$t^{s_1-1} \log^{d(1-s_1)}(e+\frac{1}{t})$. This completes the proof.
\end{proof}

\begin{Corollary}\label{C:cor}
Let $0<s_1 \leq s_2\leq \dots \leq s_n <1$, where exactly $d$ of the numbers $s_2,\dots,s_n$ are equal to $s_1$. 
Assume that $E$ is a~rearrangement-invariant space such that $0<\alpha_E\leq \beta_E <1$ and
\begin{equation}\label{E:E}
\|T_\sigma f\|_{L^p(\R^n)} \lesssim \sup_{j_1,\dots , j_n\in \mathbb{Z}} \Norm{\Gamma \left(s_1,\dots , s_n\right)\big[\wh{\Psi}\,
D_{j_1,\dots , j_n} \sigma \big]}{E(\mathbb{R}^n)} \|f\|_{L^p(\R^n)}.
\end{equation}
Then $E(\Omega) \hookrightarrow \Lambda_{\phi_{s_1,(1-s_1)d}}(\Omega)$ for all sets $\Omega \subseteq \R^n$ of finite Lebesgue 
measure.
\end{Corollary}

\begin{proof}
Assume that $\Phi$ is a smooth function on $\R^n$ with compactly supported Fourier transform and $a_1, \dots, a_n$ are fixed 
integers. We recall the estimates
\begin{equation}\label{E:first1}
\|\Gamma(s_1,\dots,s_n)[\wh{\Phi} F]\|_{E(\R^n)} \lesssim \|\Gamma(s_1,\dots,s_n) F\|_{E(\R^n)}
\end{equation}
and
\begin{equation}\label{E:second1}
\|\Gamma(s_1,\dots,s_n)[D_{a_1,\dots,a_n} F]\|_{E(\R^n)} \lesssim \|\Gamma(s_1,\dots,s_n) F\|_{E(\R^n)}
\end{equation}
for any function $F$ on $\R^n$. To verify~\eqref{E:first1} and~\eqref{E:second1} we first observe that they hold in the special 
case when $E=L^q$, $1<q<\infty$. Then we choose $1<q_1,q_2<\infty$ such that $1/q_1<\alpha_E \leq \beta_E <1/q_2$, and the 
conclusion follows by interpolating between the $L^{q_1}$ and $L^{q_2}$ endpoints via Boyd's interpolation 
theorem~\cite[Theorem 1]{Boyd} (see also the beginning of Section~\ref{S:interpolation} for the statement of this theorem).

Let us consider testing functions $\sigma$ of the form
\begin{equation}\label{E:sigma_definition}
\sigma=[(G_{s_1} \otimes \cdots \otimes G_{s_n})* g]\wh{\eta},
\end{equation}
where $\eta$ is a smooth function on $\R^n$ satisfying $\wh{\eta}=1$ on the cube $[7/8,9/8]^n$ and such that the support of 
$\wh{\eta}$ is contained in $[3/4,5/4]^n$. Taking into account the support properties of $\wh{\Psi}$ we deduce that 
$\wh{\Psi}\,D_{j_1,\dots , j_n} \sigma=0$ unless $j_i\in \{-1,0,1\}$ for each $i=1,2,\dots,n$. Inequality~\eqref{E:E} 
combined with the fact that $\|T_\sigma\|_{L^p(\R^n) \rightarrow L^p(\R^n)} \gtrsim \|\sigma\|_{L^\infty(\R^n)}$ yields
\begin{equation}\label{E:ee}
\|\sigma\|_{L^\infty(\R^n)}
\lesssim \sup_{j_1,\dots , j_n\in \{-1,0,1\}} \Norm{\Gamma \left(s_1,\dots , s_n\right)\big[  \wh{\Psi}\,
D_{j_1,\dots , j_n} \sigma \big]}{E(\mathbb{R}^n)}.
\end{equation}
Using~\eqref{E:first1} and~\eqref{E:second1}, this implies
\[
\|\sigma\|_{L^\infty(\R^n)} \lesssim \Norm{\Gamma \left(s_1,\dots , s_n\right)
 \sigma }{E(\mathbb{R}^n)}.
\]
An application of~\eqref{E:sigma_definition} and~\eqref{E:first1} then gives
\[
\|[(G_{s_1} \otimes \cdots \otimes G_{s_n})* g]\wh{\eta}\|_{L^\infty(\R^n)} \lesssim \Norm{g}{E(\mathbb{R}^n)}.
\]
Since $\wh{\eta}=1$ on $[7/8,9/8]^n$, the proof of Proposition~\ref{P:optimality} applied with 
$(\tilde{x}_1,\dots,\tilde{x}_n) \in [7/8,9/8]^n$ yields the conclusion.
\end{proof}

\begin{Example}\label{example}
We apply Corollary~\ref{C:cor} with the Lorentz space $E=\Lambda_{\phi_{s_1,\beta}}$, where $\beta \in \R$ (note that $\alpha_{\Lambda_{\phi_{s,\beta}}}=\beta_{\Lambda_{\phi_{s,\beta}}}=s_1$). Thus, a necessary condition for inequality
\begin{equation}\label{E:lorentz-example}
\|T_\sigma f\|_{L^p(\R^n)} \lesssim \sup_{j_1,\dots , j_n\in \mathbb{Z} } \Norm{\Gamma \left(s_1,\dots , s_n\right)\big[  \wh{\Psi}\,
D_{j_1,\dots , j_n} \sigma \big]}{\Lambda_{\phi_{s_1,\beta}}(\mathbb{R}^n)} \|f\|_{L^p(\R^n)}
\end{equation}
to be satisfied is the validity of embedding $\Lambda_{\phi_{s_1,\beta}}(\Omega) \hookrightarrow \Lambda_{\phi_{s_1,(1-s_1)d}}(\Omega)$ 
for all sets $\Omega \subseteq \R^n$ of finite measure. This is equivalent to the pointwise estimate $\phi_{s_1,(1-s_1)d}(t) \lesssim \phi_{s_1,\beta}(t)$ for $t$ near $0$ (see, e.g., \cite[Theorem 10.3.8]{PKJF}), which in turn yields the explicit necessary condition $\beta \geq (1-s_1)d$. 
In particular, $\beta=0$ is not allowed unless $d=0$, and estimate
\begin{equation*}
\|T_\sigma f\|_{L^p(\R^n)} \lesssim \sup_{j_1,\dots , j_n\in \mathbb{Z} } \Norm{\Gamma \left(s_1,\dots , s_n\right)\big[  \wh{\Psi}\,
D_{j_1,\dots , j_n} \sigma \big]}{L^{\frac{1}{s_1},1}(\mathbb{R}^n)} \|f\|_{L^p(\R^n)}
\end{equation*}
thus fails whenever at least one of the indices $s_2, \dots, s_n$ equals $s_1$. On the other hand, if $s_1\leq 1/2$ then we will 
prove that condition~\eqref{E:lorentz-example} is satisfied whenever $\beta>(1-s_1)d$, see Remark~\ref{R:remark_logarithm} below.
\end{Example}


\section{The core of the proof}\label{S:core}

In this section we prove Theorem~\ref{1/2MainResult} in the special case when $1/2<s_1 <1$. The general case then follows by 
interpolation; the details can be found in Section~\ref{S:interpolation}. We point out that   in fact we prove a~slightly stronger 
variant of Theorem~\ref{1/2MainResult} in this particular case; namely, we   replace the constant $K$ in~\eqref{hypoK} by the smaller 
constant
\[
\widetilde{K}=	\sup_{j_1,\dots , j_n\in \mathbb{Z} } \Norm{\Gamma \left(s_1,\dots , s_n\right)\big[  \wh{\Psi}\,
D_{j_1,\dots , j_n} \sigma \big]}{\Lambda_{\phi_{s_1, s_1d}}(\mathbb{R}^n)}.
\]

\begin{proof}[Proof of Theorem $\ref{1/2MainResult}:$ case $1/2<s_1 <1$]
Given a Schwartz function $\psi$ as in the statement of the theorem,  we define
a new Schwartz function $\psi_b$ ($\psi$ big) on $\mathbb{R}$ as follows:
\begin{equation}\label{defpsib}
\widehat{\psi_b}(\xi)= \widehat{\psi}(\xi/2)+\widehat{\psi}(\xi)+\widehat{\psi}(2\xi)\,.
\end{equation}
Then $\widehat{\psi_b}$ is supported in the annulus $1/4<|\xi| <4 $ and $\widehat{\psi_b}=1 $ on the support of $\wh{\psi}$.
Recalling the definition of $\Psi$ given in \eqref{defPsi}, we  introduce  a Schwartz function $\Psi_b$ on $\mathbb{R}^n$  
by setting
\[
 \widehat{\Psi_b}  =
 \overbrace{\widehat{\psi_b}\otimes \cdots \otimes   \widehat{\psi_b}\,\,}^{\textup{$n$ times }  }  .
\]
For $j\in \mathbb{Z}$ we can define the Littlewood-Paley operators corresponding to $\psi$ and $\psi_b$ in the  $k$th variable 
as the operators whose action on a function $f$ on $\rn$ is as follows:
\[
\Delta_{j}^{\psi, k}(f)(x_1,\dots, x_n)=\int_{\mathbb{R}} f\left(\dots,  x_{k}-y, \dots  \right) 2^{j} \psi\left(2^{j} y\right) dy
\]
and
\begin{equation*}
\Delta_{j}^{\psi_b, k}(f)(x_1,\dots, x_n)
=\int_{\mathbb{R}} f\left(\dots,  x_{k}-y, \dots \right) 2^{j} \psi_b\left(2^{j} y\right) d y.
\end{equation*}
Since $\widehat{\psi_b}=1$ on the support of $\widehat{\psi}$, $\widehat{\Psi_b}(2^{-j_1}\xi_1,\dots , 2^{-j_n}\xi_n)=1$ on the 
support of $\widehat{\Psi}(2^{-j_1}\xi_1,\dots, 2^{-j_n}\xi_n) $ for each $j_1,\dots, j_n \in \mathbb{Z}$ and so 	
\begin{align*}
		&\Delta_{j_1}^{\psi,1} \cdots \Delta_{j_n}^{\psi,n} T_{\sigma}(f)\left(x_1,\dots, x_n\right) \\
		&=\int_{\rn}  \widehat{f}(\xi_1,\dots \xi_n) \widehat{\Psi}\left(2^{-j_1} \xi_1,
		\dots, 2^{-j_n} \xi_n\right) \sigma(\xi_1,\dots, \xi_n) e^{2 \pi i (x_1 \xi_1+\cdots +  x_n \xi_n)} d \xi_1\cdots d \xi_n\\
		&=\int_{\rn}   \widehat{f}(\xi_1,\dots \xi_n)  \widehat{\Psi_b}\left(2^{-j_1} \xi_1,\dots,
		2^{-j_n} \xi_n \right) \widehat{\Psi}\left(2^{-j_1} \xi_1,\dots  ,  2^{-j_n} \xi_n\right) \\
		& \hspace{6cm}\sigma(\xi_1,\dots, \xi_n) e^{2 \pi i (x_1 \xi_1+ \cdots + x_n \xi_n)} d \xi_1 \cdots d \xi_n\\
		&=\int_{\rn}  (\Delta_{j_1}^{\psi_b,1}\cdots  \Delta_{j_n}^{\psi_b,n}f)\;\widehat{\;}\,(\xi_1,\dots , \xi_n) \widehat{\Psi}\left(2^{-j_1} \xi_1,\dots, 2^{-j_n} \xi_n\right) \\
		& \hspace{6cm}\sigma(\xi_1,\dots, \xi_n) e^{2 \pi i (x_1 \xi_1+ \cdots + x_n \xi_n)} d \xi_1\cdots d \xi_n\\
		&=\int_{\rn}   2^{j_1+\cdots + j_n} (\Delta_{j_1}^{\psi_b,1} \cdots \Delta_{j_n}^{\psi_b,n}f)\;\widehat{\;}\,(2^{j_1}\xi'_1,\dots , 2^{j_n}\xi'_n)\\
		& \hspace{3cm} \widehat{\Psi}\left(\xi'_1, \dots, \xi'_n\right) \sigma(2^{j_1}\xi'_1,\dots, 2^{j_n}\xi'_2) e^{2 \pi i (2^{j_1}x_1 \xi'_1+ \cdots + 2^{j_n} x_n \xi'_n)} d \xi'_1 \cdots d \xi'_n\\
		&= \int_{\rn}   (\Delta_{j_1}^{\psi_b,1} \cdots \Delta_{j_n}^{\psi_b,n}f)(2^{-j_1}y'_1,\dots, 2^{-j_n}y'_n)\\
		&\hspace{4cm} \big[\widehat{\Psi}
	D_{j_1,\dots , j_n} \sigma \big]\;\widehat{\;}\,(y'_1-2^{j_1}x_1, \dots, y'_n-2^{j_n}x_n )  d y'_1 \cdots d y'_n\\
		&=\int_{\rn}   (\Delta_{j_1}^{\psi_b,1} \cdots \Delta_{j_n}^{\psi_b,n}f)(2^{-j_1}y_1+x_1,\dots, 2^{-j_n}y_n+x_n)\\
		& \hspace{6cm} [\widehat{\Psi} D_{j_1,\dots , j_n} \sigma ]\;\widehat{\;}\, (y_1, \dots, y_n )  d y_1 \cdots d y_n\\
		&=\int_{\rn}   \dfrac{\big(\Delta_{j_1}^{\psi_b,1}\cdots \Delta_{j_n}^{\psi_b,n}f\big)(2^{-j_1}y_1+x_1,\dots,
		2^{-j_n}y_n+x_n)}{(1+|y_1|)^{s_1} \cdots (1+|y_n|)^{s_n}} \\
		& \hspace{2cm} (1+|y_1|)^{s_1} \cdots (1+|y_n|)^{s_n}
		[\widehat{\Psi} D_{j_1,\dots , j_n} \sigma ]\;\widehat{\;}\,(y_1, \dots, y_n ) \, d y_1\cdots d y_n.
	\end{align*}
	
Applying   H\"older's inequality in the Lorentz-Marcinkiewicz setting~\eqref{E:holder}, we obtain that
$ \lvert \Delta_{j_1}^{\psi,1} \cdots \Delta_{j_n}^{\psi,n} T_{\sigma}(f)\left(x_{1},\dots, x_n\right)\rvert$ is bounded by
\begin{equation*}
\begin{aligned}
&\Norm{\dfrac{\big(\Delta_{j_1}^{\psi_b,1} \cdots \Delta_{j_n}^{\psi_b,n}f\big)(2^{-j_1}y_1+x_1,\dots, 2^{-j_n}y_n+x_n)}{(1+|y_1|)^{s_1}\cdots  (1+|y_n|)^{s_n}}}{{M_{\om_{s_1,-s_1d}}(\mathbb{R}^n,dy_1\cdots dy_n)}}\\
&  \cdot \Norm{(1+|y_1|)^{s_1} \cdots (1+|y_n|)^{s_n}
[\widehat{\Psi} D_{j_1,\dots , j_n} \sigma ]\;\widehat{\;}\, (y_1, \dots, y_n )}{{\Lambda_{\om_{1-s_1,s_1d}}(\mathbb{R}^n)}}\,.
\end{aligned}
	\end{equation*}
The first term in this product is estimated by Corollary~\ref{cor88}  as follows: Since we are assuming $1>s_1>1/2$,
there is a $q$  such that $1< {1}/{s_1}<q<2$. Then for this $q$ we get
	\begin{align*}
		\begin{split}
			&\Norm{\dfrac{\big(\Delta_{j_1}^{\psi_b,1} \cdots \Delta_{j_n}^{\psi_b,n}f\big)(2^{-j_1}y_1+x_1,\dots, 2^{-j_n}y_n+x_n)}{(1+|y_1|)^{s_1}\cdots (1+|y_n|)^{s_n}}}{{M_{\om_{s_1,-s_1d}}}(\mathbb{R}^n)} \\ &\hspace{4cm}\leq C \mathcal{M}_{{L}^q}\big(\lvert \Delta_{j_1}^{\psi_b,1} \cdots \Delta_{j_n}^{\psi_b,n} f  \rvert \big) (x_1,\dots, x_n) .
		\end{split}
	\end{align*}

We estimate the second term in the product using Proposition~\ref{phiomega} and Lemma~\ref{lorentz}, i.e., the  Hausdorff-Young inequality adapted to these Lorentz spaces.    We obtain
\begin{equation}
		\begin{aligned}
			&\Norm{ \prod_{i=1}^n (1+|y_i|)^{s_i}   [\widehat{\Psi} D_{j_1,\dots, j_n} \sigma ]
\;\widehat{\;}\, (y_1,\dots,  y_n )}{{\Lambda_{\om_{1-s_1,s_1d}}   (\mathbb{R}^n)}}\notag \\
			&\leq C\Norm{ \prod_{i=1}^n (1+|y_i|^2)^{\f {s_i}2}  [\widehat{\Psi} D_{j_1,\dots, j_n} \sigma ]\;\widehat{\;}\, (y_1, \dots, y_n )}
			{\Lambda_{\om_{1-s_1,s_1d}}(\mathbb{R}^n)  } \notag \\
			&\leq C \Norm{\Gamma\left( s_1 , \dots, s_n \right) [\widehat{\Psi} D_{j_1,\dots, j_n} \sigma ]}{\Lambda_{\phi_{s_1,s_1d}}(\mathbb{R}^n)}\notag
			\leq C\widetilde{K}. \notag
		\end{aligned}
	\end{equation}
where we used the fact that $1<1/s_1<2$, which is
	a hypothesis of  Lemma~\ref{lorentz}.

We have now obtained the pointwise estimate
	\begin{equation}\lab{PW}
\lvert \Delta_{j_1}^{\psi,1}\cdots  \Delta_{j_n}^{\psi,n} T_{\sigma}(f) \rvert \leq C \widetilde{K} \mathcal{M}_{{L}^q}\big(\lvert \Delta_{j_1}^{\psi_b,1}\cdots  \Delta_{j_n}^{\psi_b,n} f \rvert\big).
	\end{equation}
	
	Now let $p\geq2$. Applying the  product type Littlewood-Paley theorem, the Fefferman-Stein inequality,
	and estimate \eqref{PW} we obtain
\begin{align*}
		\left\|T_{\sigma}(f)\right\|_{L^{p} \left(\mathbb{R}^{n}\right)} &\leq C\left\|\left(\sum_{j_1,\dots, j_n \in \mathbb{Z}}\left|\Delta_{j_1}^{\psi,1} \cdots \Delta_{j_n}^{\psi,n} T_{\sigma}(f)\right|^{2}\right)^{\frac{1}{2}}\right\|_{L^{p}\left(\mathbb{R}^{n}\right)}\\
		&\leq C\widetilde{K}\left\|\left(\sum_{j_1,\dots, j_n \in \mathbb{Z}}\left|  \mathcal{M}_{{L}^q}\left(\lvert \Delta_{j_1}^{\psi_b,1}  \cdots\Delta_{j_n}^{\psi_b,n} f \rvert\right) \right|^{2}\right)^{\frac{1}{2}}\right\|_{L^{p}\left(\mathbb{R}^{n}\right)}\\
		&\leq C\widetilde{K}\left\|\left(\sum_{j_1,\dots, j_n \in \mathbb{Z}} \left( \mathcal{M} \Big( \lvert \Delta_{j_1}^{\psi_b,1}  \cdots\Delta_{j_n}^{\psi_b,n} f \rvert^{q}\Big)\right)^{\frac{2}{q}}\right)^{\frac{1}{2}}\right\|_{L^{p}\left(\mathbb{R}^{n}\right)}\\
		&\leq C\widetilde{K}\left\|\left(\sum_{j_1,\dots, j_n \in \mathbb{Z}} \left( \mathcal{M} \Big( \lvert \Delta_{j_1}^{\psi_b,1}  \cdots\Delta_{j_n}^{\psi_b,n} f \rvert^{q}\Big)\right)^{\frac{2}{q}}\right)^{\frac{q}{2}}\right\|_{L^{\frac{p}{q}}\left(\mathbb{R}^{n}\right)}^\frac{1}{q}\\
		&\leq C'\widetilde{K}\left\|\left(\sum_{j_1,\dots, j_n \in \mathbb{Z}} \lvert \Delta_{j_1}^{\psi_b,1}  \cdots \Delta_{j_n}^{\psi_b,n} f \rvert^{q.\frac{2}{q}}\right)^{\frac{q}{2}}\right\|_{L^{\frac{p}{q}}\left(\mathbb{R}^{n}\right)}^\frac{1}{q}\\
		&\leq C''\widetilde{K}\left\|\left(\sum_{j_1,\dots, j_n \in \mathbb{Z}}  \lvert \Delta_{j_1}^{\psi_b,1} \cdots \Delta_{j_n}^{\psi_b,n} f \rvert^{2}\right)^{\frac{1}{2}}\right\|_{L^{p}\left(\mathbb{R}^{n}\right)}\\
		&\leq C'''\widetilde{K} \Norm{f}{{L}^p(\mathbb{R}^n)}\,.
	\end{align*}
The case   $1<p<2$ follows by duality.
\end{proof}

\section{Interpolation}\label{S:interpolation}

In the previous section we have proved the main theorem under the extra assumption that  $1/2<s_1<1$. This estimate will be useful for $p$ near $1$
or near $\infty$ while for $p=2$ we can use the trivial $L^\infty$ estimate for the multiplier. The final conclusion will be
a~consequence of an   interpolation result (Theorem~\ref{Interpolation}) discussed in this section.

We start with a few lemmas. In the proof we will use Boyd's interpolation theorem (see \cite[Theorem 1]{Boyd}) which states:
If $E$ is an~r.i. space on $\mathbb{R}_+$ such that $1/p_1 <\alpha_E  \leq \beta_E <1/p_0$ for some $1<p_0<p_1<\infty$, then the r.i.
space $E(\Omega)$ on $(\Omega, \mu)$ is interpolation between $L^{p_0}(\mu)$ and $L^{p_1}(\mu)$, i.e., $L^{p_0}(\mu) \cap  L^{p_1}(\mu)
\hookrightarrow E(\Omega) \hookrightarrow L^{p_0}(\mu) + L^{p_1}(\mu)$ and for any linear operator $T$ on $L^{p_0}(\mu) + L^{p_1}(\mu)$
such that $T$ is bounded on $L^{p_j}(\mu)$ for $j=0$ and $j=1$, it follows that $T$ is a~bounded operator on $E(\Omega)$.

\begin{Lemma}\label{KatoPonce}
Let $\Phi$ be a smooth function on $\R^n$ with compactly supported Fourier transform. Then, for any $0<s, s_2\dots,s_n<1$, $\ga>0$ and
any function $F$ on $\rn$, we have
\[
\Norm{ \Gamma\left( s ,s_2,\dots,  s_n \right) \big[ \widehat{\Phi} \,F  \big] }{\Lambda_{\phi_{s, \ga}} (\R^n) }  \le
\Norm{ \Gamma\left( s , s_2,\dots,  s_n \right) F }{\Lambda_{\phi_{s, \ga}} (\R^n)}\,.
\]
\end{Lemma}

\begin{Lemma}\label{imaginary-Lorentz}
Let $0<s<1$ and $\gamma>0$. Then, for any $t_1,\dots , t_n \in \mathbb{R}$, we have
\[
\Norm{\Gamma\left( it_1,\dots, it_n\right) f}{ \Lambda_{\phi_{s,\ga}} (\rn)} \leq
C(p,n )(1+|t_1|) \cdots (1+|t_n|) \Norm{f}{\Lambda_{\phi_{s,\ga}} (\rn)}\,.
\]
\end{Lemma}

\begin{Lemma}\label{Marc-inter}
Let $0<s<1 $. Let $m$ be a function satisfying \eqref{Eq-Marc2}. Then, for any $\ga>0$, we have
\[
\Norm{T_m (f) }{ \Lambda_{\phi_{s,\ga}} (\rn)}
\leq C(p,n )  \Norm{f}{\Lambda_{\phi_{s,\ga}}(\rn) }.
\]

\end{Lemma}

All these lemmas can be proved in the following way.

\begin{proof}
It is easy to check that if $\varphi\in \mathcal{P}$ and $\Lambda_\varphi$ is the~Lorentz space on $\mathbb{R}_+$,
then $\|\sigma_t\|_E = m_\varphi(t)$ for all $t>0$. This implies that $\alpha_{\Lambda_\varphi} = \gamma_\varphi$
and $\beta_{\Lambda_\varphi} = \delta_\varphi$.

Now we choose $p_0$ and $p_1\in (1, \infty)$ such that $1/p_1 < s< 1/p_0$ and let $E:=\Lambda_{\phi_{s, \ga}}$. Combining the above fact with Proposition
\ref{phiomega}, we conclude that
\[
1/p_1 <\alpha_E =\beta_E = s<1/p_0\,.
\]
Since the estimates hold for $L^{p_0}$ and  ${L}^{p_1}$ in place of the Lorentz space, Boyd's interpolation theorem
completes the proof.
\end{proof}

We will use the following lemma (see \cite[Lemma 2.1]{GrNg}).

\begin{Lemma}\label{lem:016}
Let    $0< p_0 \le p_1 < \infty$ and define $p$ via $1/p=(1-\theta)/p_0+\theta/p_1$, where $0 <\theta < 1$. Given
$f\in {\mathscr C}_0^\nf(\rn)$ and   $\ve>0,$ there exist smooth functions $h_j^\ve$, $j=1,\dots, N_\ve$, supported
in cubes with disjoint interiors, and there exist nonzero complex constants $c_j^\ve$ such that the functions
\begin{equation}\label{EQ-Fz}
f_{z}^{\ve}  =  \sum_{j=1}^{N_\ve} |c_j^\ve|^{\frac p{p_0} (1-z) +   \frac p{p_1}  z}  \, h_j^\ve
\end{equation}
satisfy
\begin{equation}\label{newest}
\big\|f_{\theta}^{\ve}-f\big\|_{L^{p_0}}+
\big\|f_{\theta}^{\ve}-f\big\|_{L^{p_1}}+
\big\|f_{\theta}^{\ve}-f\big\|_{L^{2}} <  \ve
\end{equation}
and
\begin{equation}\label{newest2}
\|f_{it}^{\ve}\|_{L^{p_0}}^{p_0} \le   \|f \|_{L^p}^p +\ve'  \, , \q
\|f_{1+it}^{\ve}\|_{L^{p_1}}^{p_1} \le    \|f \|_{L^p}^p  +\ve'  \, ,
\end{equation}
where $\ve'$ depends on $\ve ,p, \|f\|_{L^p}$ and tends to zero as $\ve\to 0$.
\end{Lemma}	

The next lemma is a~variant of Lemma 3.7 from \cite{GSHor20}.

\begin{Lemma}\label{Lorentz maximal}
Let $0<\al,\be<1$, $\ga> 0$. Then  for some constant $ C( \al,\be ,\ga)$ we have
\begin{equation} \label{E:sunrise}
\int_0^\infty (f^*(r) r^{ \be-\al })^*(y) \phi_{\al,\ga}(y)\, \frac{dy}{y}
\leq C( \al,\be ,\ga) \int_0^\infty f^*(r) \phi_{\be,\ga}(r)\,\frac{dr}{r}\,.
\end{equation}
\end{Lemma}

\begin{proof}
Recall that for given $s \in (0, 1)$ and $\ga>0$, we have the equivalence
$\phi_{s,\ga}(t)\approx t^s \log^\ga \big(e + \frac{1}{t}\big)$
on $(0, \infty)$. Now observe that the estimate  \eqref{E:sunrise} is trivial when $\be\le \al$
as $(f^*(r) r^{ \be-\al })^*=  f^*(r) r^{ \be-\al } $. Thus we may assume that $\be>\al$ in the
proof below. We may also assume that
\[
\int_0^\infty f^*(r) r^{\be -1} \log^{\gamma}\Big(e+\f 1r\Big) \,dr < \infty\,,
\]
otherwise the right-hand side of~\eqref{E:sunrise} is infinite.
Then
$$
\sup_{r>0} f^*(r) r^{\be } \log^\ga \Big(e + \frac{1}{r}\Big)\leq C,
$$
and thus $\lim_{r\to \infty} f^*(r) r^{\be-\al }=0$.     Since the set of discontinuity points
of $f^{*}$ is at most countable ($f^{*}$ is right continuous), we may assume without loss of generality that
function $f^{*}$ is continuous. Then $\sup_{y\leq r<\infty} f^*(r) r^{\be-\al}$ is attained for any $y>0$ and
so the set
\[
M=\{y\in (0,\infty): \sup_{y\leq r<\infty} f^*(r) r^{\be-\al} > f^*(y) y^{\be-\al}\}
\]
is open. Hence, $M$ is a countable union of open intervals, namely, $M=\bigcup_{k\in S} (a_k,b_k)$, where
$S$ is a~countable set of positive integers. Also, observe that if $y\in (a_k,b_k)$, then
\[
\sup_{y\leq r<\infty} f^*(r) r^{\be-\al} 
=f^*(b_k) b_k^{\be-\al}\,.
\]
We have
\begin{align*}
& \hspace{-.2in} \int_0^\infty (f^*(r) r^{\be-\al})^*(y) y^{\al -1} \log^{\ga}\Big( e+\f 1y\Big) \,dy \\
&\leq \int_0^\infty \sup_{y\leq r<\infty} f^*(r) r^{\be-\al} y^{\al-1} \log^{\ga}\Big( e+\f 1y\Big)  \,dy\\
&\le \int_{(0,\infty)\setminus M} f^*(y) y^{\be-1} \log^{\ga}\Big( e+\f 1y\Big) \,dy\\
&\hspace{1in}+\sum_{k\in S} f^*(b_k) b_k^{\be-\al} \int_{a_k}^{b_k} y^{\al-1} \log^{\ga}\Big( e+\f 1y\Big) \,dy\,.
\end{align*}
Furthermore, for every $k\in S$,
\begin{align*}
&f^*(b_k) b_k^{\be-\al} \int_{a_k}^{b_k} y^{\al-1} \log^{\ga} \Big( e+\f 1y\Big) \,dy \\
&\leq f^*(b_k) b_k^{\be-\al}\int_{\max(a_k,\frac{b_k}{2})}^{b_k} y^{\al-1} \log^{\ga}\Big( e+\f 1y\Big)\,dy \cdot
\frac{\int_0^{b_k} y^{\al-1} \log^{\ga}( e+\f 1y)\,dy}{\int_{\frac{b_k}{2}}^{b_k} y^{\al-1} \log^{\ga}( e+\f 1y )\,dy}\\
&=C_{\al,\ga}  f^*(b_k) b_k^{\be-\al} \int_{\max(a_k,\frac{b_k}{2})}^{b_k} y^{\al-1}
\log^{\ga}\Big( e+\f 1y\Big) \, dy\\
&\leq C_{\al,\be, \ga}  \int_{a_k}^{b_k} f^*(y) y^{\be-1} \log^{\ga}\Big( e+\f 1y\Big)  \,dy\,.
\end{align*}
Therefore,
\begin{align*}
& \int_0^\infty (f^*(r) r^{\be-\al})^*(y) y^{\al-1} \log^{\ga}\Big( e+\f 1y\Big) \,dy \\
&\leq \int_0^\infty f^*(y) y^{\be-1} \log^{\ga} \Big( e+\f 1y\Big)\,dy
+ C_{\al,\be,\ga}\sum_{k\in S} \int_{a_k}^{b_k} f^*(y) y^{\be-1} \log^{\ga}\Big( e+\f 1y\Big) \,dy\\
&\leq  (C_{\al,\be,\ga}+1) \int_0^\infty f^*(y) y^{\be-1} \log^{\ga} \Big( e+\f 1y\Big) \,dy\,.
\end{align*}
This proves \eqref{E:sunrise}.
\end{proof}

The main interpolation tool in this work is the following.

\begin{Theorem}\label{Interpolation}
Let $1<p_0<\infty $ and  suppose that $\frac{1}{2}<s_1^0 \le s_2^0 \le \cdots \le s_n^0<1$ and that
$0< s_{1}^1\le s_2^1\le \cdots \le s_n^1  < 1$. Assume that exactly $d$ of the numbers $s_2^0,\dots,s_n^0$ are equal
to $s_1^0$, and exactly $d$ of the numbers $s_2^1,\dots,s_n^1$ are equal to $s_1^1$.  Let $\Psi$ be as in \eqref{defPsi}.
Suppose that  for   all nonzero $f \in \mathscr{C}_0^{\infty} (\mathbb{R}^n)$ we have
\begin{equation}\label{E:first}
\big\|T_{\sigma} f\big\|_{L^{p_0}(\mathbb{R}^n)}
\leq  K_0 \sup_{ j_1,\dots, j_n \in \mathbb{Z}  }\big\|\Gamma\big( s_1^0  , \dots , s_n^0  \big) [ \widehat{\Psi} D_{j_1,\dots, j_n}  \sigma  ] \big\|_{\Lambda_{\phi_{s_1^0, s_1^0 d}}(\mathbb{R}^n)} \|f\|_{L^{p_{0}}(\mathbb{R}^n)}
\end{equation}
and
\begin{equation}\label{E:second}
\big\|T_{\sigma} f\big\|_{L^{p_1}(\mathbb{R}^n)}
\!\leq  \! K_1\!\!\!  \sup_{ j_1,\dots, j_n \in \mathbb{Z}  }\big\|\Gamma\big( s_1^1  , \dots , s_n^1  \big) [ \widehat{\Psi} D_{j_1,\dots, j_n}  \sigma] \big\|_{\Lambda_{\phi_{s_1^1, (1-s_1^1) d}}(\mathbb{R}^n)} \|f\|_{L^{p_{1}}(\mathbb{R}^n)}.
\end{equation}
Let   $0<\theta <1$ and suppose
\[
\frac{1}{p}= \frac{1-\theta}{p_0}+\frac{\theta}{2},\quad s_j =(1-\theta)s_j^0+ \theta s_j^1, \quad j=1,\dots , n.
\]
Then there is a constant $C_{*}=C_{*}(p_0,\theta,n,d, \psi,s_j^0,s_j^1)$ such that for all $f \in \mathscr{C}_0^{\infty}(\mathbb{R}^n)$
\begin{equation*}
\big\|T_{\sigma}f\big\|_{L^{p}(\mathbb{R}^n)} \leq C_{*} K_0^{1-\theta} K_1^{\theta}
\sup_{j_1,\dots, j_n\in \mathbb{Z} }\big\|\Gamma\left( s_1,\dots, s_n  \right)[\widehat{\Psi} D_{j_1,\dots, j_n} \sigma]\big\|_{\Lambda_{\phi_{s_1, d}} (\mathbb{R}^n)} \|f\|_{L^{p }(\mathbb{R}^n)}\,.
\end{equation*}
\end{Theorem}

\begin{proof}
Let us fix a function $\sigma$ such that
\begin{equation}\label{assump77}
\sup_{  j_1,\dots, j_n \in \mathbb{Z}   }\Norm{ \Gamma(s_1,\dots, s_n) \big[ \widehat{\Psi} \,
D_{j_1,\dots , j_n}  \sigma   \big] }{  \Lambda_{\phi_{s_1,   d}} (\rn) } < \infty
\end{equation}
and for $j_1, \dots , j_n \in \mathbb{Z}$ define
\[
\varphi_{j_1,\dots , j_n}  = \Gamma(s_1,\dots , s_n) \big[ \widehat{\Psi} \,  D_{j_1,\dots , j_n}  \sigma \big]\,.
\]
Since $\varphi_{j_1,\dots, j_n} \in \Lambda_{\phi_{s_1, d}} (\rn)$, we have $\sup_{\lambda>0} \phi_{s_1, d}(\lambda)\varphi^{*}_{j_1,\dots, j_n}(\lambda) < \infty$ and so $\varphi^{*}_{j_1,\dots, j_n}(\lambda) $ converges to $ 0$ as $\lambda \rightarrow \infty$.
Now by  \cite[Corollary 7.6 in Chapter 2]{BSInterpol88}, there is a~measure preserving transformation $h_{j_1,\dots, j_n}: \rn \rightarrow (0,\infty)$
such that
\begin{equation}\label{BSha}
|\varphi_{j_1,\dots, j_n}| = \varphi_{j_1,\dots, j_n}^* \circ h_{j_1,\dots, j_n}\,.
\end{equation}
	
Recall that $s_1^0\le \cdots \le s_n^0$ and $s_1^1\le \cdots \le  s_n^1$.  For $z \in \mathbb{C}$ with $0\leq \textup{Re}(z)\leq 1$,
we define  complex polynomials
\[
P_\rho(z) =   s_\rho^0 (1-z) +s_\rho^1\, z
\]
for $\rho=1,2,\dots , n$. 	Let $\widehat{\Psi_b}  =\widehat{\psi_b}\otimes \cdots\otimes \widehat{\psi_b} $  where $\psi_b$  is defined in \eqref{defpsib}. We define the family of multipliers
\begin{equation}\label{E:sigma}
\sigma_z  = \sum_{k_1,\dots, k_n \in \mathbb{Z}}D_{-k_1,\dots, -k_n}
\bigg[  \widehat{\Psi_b} \, \Gamma\big(-P_1(z) ,\dots, - P_n(z) \big)
\Big[ \varphi_{k_1,\dots, k_n} h_{k_1,\dots, k_n}^{s_1-P_1(z)}  \Big] \bigg]\,.
\end{equation}
As $P_j(\theta) = s_j$ for $1\le j \le n $  and
$\sum_{k_1,\dots , k_n \in \mathbb{Z}} \widehat{\Psi}\left(   2^{-k_1}\xi_1,\dots , 2^{-k_n}\xi_n  \right)=1$ when all $\xi_k\neq 0$,  it follows that    $\sigma_\theta= \sigma$ a.e.

Fix $f,g \in \mathscr{C}_0^{\infty}(\mathbb{R}^n)$. Given $\epsilon >0$ find $f_{z}^{\epsilon}$ and $g_z^{\epsilon}$ as in Lemma~\ref{lem:016}.
Thus we have $\Norm{f^{\epsilon}_{\theta}-f}{{L}^p}+\Norm{f^{\epsilon}_{\theta}-f}{{L}^2}<\epsilon$, $\Norm{g_{\theta}^{\epsilon}-g}{{L}^{p^{\prime}}} +\Norm{g_{\theta}^{\epsilon}-g}{{L}^{2}}\leq \epsilon$,
\begin{equation*}
\begin{aligned}
\Norm{f^{\epsilon}_{it}}{{L}^{p_0}(\rn)}^{p_0} &\leq  \Norm{f}{{L}^p(\rn)}^p +\epsilon'  , \quad \Norm{f^{\epsilon}_{1+it}}{{L}^{2}(\rn)}^{2} \leq   \Norm{f}{{L}^{p}(\rn)}^p +\epsilon' ,\\
\Norm{g^{\epsilon}_{it}}{{L}^{p^{\prime}_0}(\rn)}^{p_0'} &\leq   \Norm{g}{{L}^{p^{\prime}}(\rn)}^{p^{\prime}} + \epsilon'
, \quad \Norm{g^{\epsilon}_{1+it}}{{L}^{2}(\mathbb{R}^n)}^{2} \leq \Norm{g}{{L}^{p^{\prime}}(\rn)}^{p^{\prime}} + \epsilon' \,.
\end{aligned}
\end{equation*}
	
Now define  on the unit  strip $\{  z \in \mathbb{C} : 0 \le \textup{Re}(z) \le 1  \}$  the following function
	\begin{equation}\label{E:definition_F}
	F(z) = \int_{\rn} \sigma_z(\xi ) \widehat{f^{\epsilon}_{z}}(\xi ) \widehat{g_{z}^{\epsilon}}(\xi ) d\xi
	= \int_{\rn} T_{\sigma_z}(f_z^\epsilon)(x)\,  g_{z}^{\epsilon} (-x) \, dx
	\end{equation}
which is  analytic in the interior of this  strip and is continuous on its closure.   H\"older's inequality
and one hypothesis of the theorem give
	\begin{align}\begin{split}\label{RETT}
	&|F(it)|\leq \Norm{T_{\sigma_{  {   it  } }}(f^{\epsilon}_{it})}{{L}^{p_0}}
	\Norm{g^{\epsilon}_{   {   it   }   }}{{L}^{p_0^{\prime}}}  \\
	&\leq  K_0 \sup_{j_1,\dots, j_n\in \mathbb{Z}  }\Norm{ \Gamma\left( s^{0}_1 ,\dots,  s^{0}_n \right) \big[ \widehat{\Psi} \,  D_{j_1,\dots, j_n}  \sigma_{it}  \big] }{\Lambda_{\phi_{s_1^0, s_1^0 d}}}   \Norm{f^{\epsilon}_{it}}{{L}^{p_0}} \Norm{g^{\epsilon}_{it}}{{L}^{p_0^{\prime}}}.
	\end{split}
	\end{align}
Using the definition of $\sigma_z$ with $z=it$, we have
\begin{align*}
\widehat{\Psi}\, &  D_{j_1,\dots, j_n}  \sigma_{it}     \\
=&  \sum_{k_1,\dots, k_n\in \mathbb Z} \!\!   \widehat{\Psi} \,D_{j_1-k_1,\dots, j_n-k_n}\Bigg[  \widehat{\Psi_b}  \, \Gamma\big(-P_1(it) ,\dots, - P_n(it) \big)
	\Big[ \varphi_{k_1,\dots, k_n} h_{k_1,\dots, k_n}^{s_1-P_1(it)}  \Big]   \Bigg] .
\end{align*}
In view of the support properties of the bumps $\wh{\Psi}$ and $ \wh{\Psi_b}$, all
terms in the sum above are  zero if  $k_i\notin \{j_i-2, j_i-1,j_i, j_i+1,j_i+2\}$  for some   $i\in\{1,\dots, n\}$.
Using   this observation and Lemma~\ref{KatoPonce} with $\wh{\Phi} = \widehat{\Psi}   \,  D_{a_1,\dots , a_n} \widehat{\Psi_b} $, we write
\begin{align*}
	&\hspace{-.16in} \Norm{ \Gamma\left( s^{0}_1 ,\dots,  s^{0}_n \right) \big[ \widehat{\Psi} \,  D_{j_1,\dots, j_n}  \sigma_{it}  \big] }{\Lambda_{\phi_{ s_1^0, s_1^0 d}} } \\
 \le & \sum_{\substack{ 1\le i \le n \\  -2\le a_i\le 2}  } \bigg\Vert  \Gamma\left( s^{0}_1 ,\dots ,  s^{0}_n \right)
 \bigg[ \widehat{\Psi}   \,  \big( D_{a_1,\dots , a_n} \widehat{\Psi_b} \big)    \\
  & \qq\q
 D_{a_1,\dots , a_n} \Big\{
	\Gamma\big( - P_1(it)  , \dots, - P_n(it) \big) \big( \varphi_{j_1+a_1,\dots, j_n+a_n} h_{j_1+a_1,\dots, j_n+a_n}^{s_1-P_1(it)} \big) \Big\}    \bigg]\bigg\Vert_{\Lambda_{\phi_{s_1^0, s_1^0 d}} }\\
	 \le & \sum_{\substack{ 1\le i \le n \\  -2\le a_i\le 2}  } \bigg\Vert  \Gamma\left( s^{0}_1 ,\dots ,  s^{0}_n \right)
	 D_{a_1,\dots , a_n} \Big\{  \Gamma\left( -s^{0}_1 ,\dots ,  -s^{0}_n \right)   \\
  & \qq\q     \Gamma\big( it ( s_1^0- s_1^1) , \dots, it( s_n^0-s_n^1) \big)
  \big( \varphi_{j_1+a_1,\dots, j_n+a_n} h_{j_1+a_1,\dots, j_n+a_n}^{s_1-P_1(it)} \big) \Big\} \bigg\Vert_{\Lambda_{\phi_{s_1^0, s_1^0 d}} }\\
	\le & C \sum_{\substack{ 1\le i \le n \\  -2\le a_i\le 2}  }  \bigg\Vert\Gamma\big( it ( s_1^0- s_1^1) , \dots, it( s_n^0-s_n^1) \big)
	\left[ \varphi_{j_1+a_1,\dots, j_n+a_n} h_{j_1+a_1,\dots, j_n+a_n}^{s_1-P_1(it)} \right]
	 \bigg\Vert_{\Lambda_{\phi_{s_1^0, s_1^0 d}} } ,
\end{align*}
as $- P_j(it) = - s_j^0+ it ( s_j^0- s_j^1)$. In the last inequality we made use of    the fact that the function
$$
\prod_{i=1}^n \bigg(  \f{ 1+ 4 \pi^2 |\xi_i|^2  }{1+ 4 \pi^2 | \xi_i /2^{a_i}|^2  } \bigg)^{  {s_i^0}/{2} }
$$
satisfies \eqref{Eq-Marc2} and thus Lemma~\ref{Marc-inter} applies.
We continue estimating as follows:
\begin{align*}
  C \sum_{\substack{ 1\le i \le n \\  -2\le a_i\le 2}  } \bigg\Vert\Gamma\big(& it ( s_1^0- s_1^1) , \dots, it( s_n^0-s_n^1) \big)  \left[ \varphi_{j_1+a_1,\dots, j_n+a_n} h_{j_1+a_1,\dots, j_n+a_n}^{s_1-P_1(it)} \right]
	 \bigg\Vert_{  \Lambda_{\phi_{s_1^0,s_1^0 d}}    (\rn)}  \\
	  \le& C \sum_{\substack{ 1\le i \le n \\  -2\le a_i\le 2}  }  (1+|t |)^n \Norm{\varphi_{j_1+a_1,\dots, j_n+a_n}
	h_{j_1+a_1,\dots, j_n+a_n}^{s_1-s_1^0} }{   \Lambda_{\phi_{s_1^0,s_1^0 d}}     (\rn)}\\
	= & C(1+|t |)^n \sum_{\substack{ 1\le i \le n \\  -2\le a_i\le 2}  }
	\Norm{\varphi^*_{j_1+a_1,\dots, j_n+a_n}(r) r^{s_1-s_1^0} }{
	\Lambda_{\phi_{s_1^0,s_1^0 d}}((0,\infty),dr)}\\
	\leq & C(1+|t |)^n \sum_{\substack{ 1\le i \le n \\  -2\le a_i\le 2}  } \Norm{\varphi^*_{j_1+a_1,\dots, j_n+a_n} }{   \Lambda_{\phi_{s_1 ,s_1^0 d}}    (0,\infty) }\\
  =  & C(1+|t |)^n \sum_{\substack{ 1\le i \le n \\  -2\le a_i\le 2}  } \Norm{\varphi_{j_1+a_1,\dots, j_n+a_n} }{ \Lambda_{\phi_{s_1 ,s_1^0 d}}(\rn)} \\
  \le  & C(1+|t |)^n \sup_{j_1,\dots, j_n\in \mathbb Z } \Norm{\varphi_{j_1,\dots, j_n} }{ \Lambda_{\phi_{s_1 , d}} (\rn)} ,
\end{align*}
where we used successively  Lemma~\ref{imaginary-Lorentz},
the fact that $ \textup{Re }  P_1(it) =s_1^0 $, identity \eqref{BSha} together with
the fact that  $h_{j_1,\dots , j_n}$ is measure-preserving, and
Lemma~\ref{Lorentz maximal}. Inserting this estimate in    \eqref{RETT} and using Lemma~\ref{lem:016}  we obtain
	\begin{equation*}
	|F(it)|   \leq
	 CK_0(1+|t |)^n \sup_{j_1,\dots, j_n \in \mathbb{Z}} \Norm{\varphi_{j_1,\dots, j_n} }
	 { \Lambda_{\phi_{s_1 , d} }       (\rn)}
\big( \!\Norm{f}{{L}^p }^p +\epsilon' \big)^{\frac{1}{p_0}}  \big( \! \Norm{g}{{L}^{p^{\prime}} }^{p^{\prime}} +\epsilon' \big)^{\frac{1}{p^{\prime}_0}}.
	\end{equation*}
A similar argument    using the inequality
\[
\Norm{\varphi_{j_1,\dots, j_n} }{ \Lambda_{\phi_{s_1 ,(1-s_1^1)d}} (\rn)}
  \le \Norm{\varphi_{j_1,\dots, j_n} }{ \Lambda_{\phi_{s_1 , d}} (\rn)}
\]
yields
	\begin{equation*}
	|F(1+it)| \leq  CK_1(1+|t |)^n \sup_{j_1,\dots, j_n \in \mathbb{Z}}\Norm{\varphi_{j_1,\dots, j_n} }
	 { \Lambda_{\phi_{s_1 , d}} (\rn)}  \big( \!\Norm{f}{{L}^{p} }^p +\epsilon' \big)^{\frac{1}{2}}
	 \big(\! \Norm{g}{{L}^{p^{\prime}} }^{p^{\prime}} +\epsilon' \big)^{\frac{1}{2}}.
	\end{equation*}
Moreover, for $\tau\in [0,1]$, we claim that $|F(\tau+it)|\le A_{\tau}(t) $ where  $A_{\tau}(t)$ has at most polynomial growth as $|t|\to \infty$;  we   prove this assertion at the end. Thus we can apply Hirschman's lemma (\cite[Lemma 1.3.8]{CFA14}).
Using the estimates for $| F(it) | $ and $|F(1+it)|$,     for $\theta \in (0,1)$,  we obtain
\begin{equation*}
	\begin{aligned}
	|F(\theta)| &\leq  C_* K_0^{1-\theta} K_1^{\theta} \sup_{j_1,\dots, j_n \in \mathbb{Z}}\Norm{\varphi_{j_1,\dots, j_n} } { \Lambda_{\phi_{s_1 , d}}       (\rn)}\big( \Norm{f}{{L}^{p} }^p +\epsilon' \big)^{\frac{1}{p}} \big( \Norm{g}{{L}^{p^{\prime}} }^{p^{\prime}} +\epsilon' \big)^{\frac{1}{p^{\prime}}}.
	\end{aligned}
	\end{equation*}
We write
\begin{equation*}
\begin{aligned}
  \left| F(\theta) -   \int_{\rn} \widehat{T_\sigma (f)} \,  \widehat{ g} (\xi)\,    d\xi \right|
	=& \left|\int_{\mathbb R^n} \sigma (\xi)\widehat{f_{\theta}^{\epsilon}}
	  (\xi)\widehat{g_{\theta}^{\epsilon} } (\xi)\,   d\xi
	- \int_{\mathbb R^n} \sigma (\xi)
	\widehat{f}(\xi) \widehat{ g} (\xi)\, d\xi \right|  \\
 \leq &    \big\|\sigma\|_{L^\nf} \Big[ \big\| f_{\theta}^{\epsilon}-f\big\|_{L^2} \big\| g\big\|_{L^2}+
\big\| g_{\theta}^{\epsilon}-g\big\|_{L^2} \big\| f\big\|_{L^2}	 \Big],
\end{aligned}
\end{equation*}
which tends to zero as $	\epsilon\to 0 $ (which implies $\epsilon' \rightarrow 0$). Thus
\[
\bigg|  \int_{\rn} \widehat{T_\sigma (f)} \,  \widehat{ g} (\xi)\,   d\xi\bigg| \le
C_* K_0^{1-\theta} K_1^{\theta} \sup_{j_1,\dots, j_n \in \mathbb{Z}}\Norm{\varphi_{j_1,\dots, j_n} }
 { \Lambda_{\phi_{s_1 , d}}(\rn)} \Norm{f}{{L}^{p} }  \Norm{g}{{L}^{p^{\prime}} } .
\]
But the integral on the left is equal to  $\int_{\rn} T_\sigma(f)(x) g(-x) dx$.
Taking the supremum over all functions $g \in  \mathscr{C}_0^{\infty}(\rn)$  with $\Norm{g}{{L}^{p^{\prime}}} \leq 1$
we deduce for $f \in \mathscr{C}_0^{\infty}(\rn)$:
\[
\Norm{T_{\sigma}(f)}{{L}^p(\rn)} \leq C_* K_0^{1-\theta} K_1^{\theta}
\sup_{j_1,\dots, j_n \in \mathbb{Z}}\Norm{\varphi_{j_1,\dots, j_n}} {\Lambda_{\phi_{s_1 , d}} (\rn)}\Norm{f}{{L}^{p}(\rn)} .
\]
Notice that the   constant $C_*$ depends on   the  parameters indicated in the statement.

We now return to the assertion  that $|F(\tau+it)|\le A_{\tau}(t) $, where  $A_{\tau}(t)$ has at most polynomial growth in $|t|$,
which was one of the hypotheses in Hirschman's lemma. Let $z= \tau + it$ where $t \in \mathbb{R}$ and $0\leq \tau \leq 1$. We use that
\[
|F(\tau+it)| \le \|\sigma_{\tau + it}  \|_{L^\infty} \| f_{ \tau + it}  \|_{L^2}  \| g_{ \tau + it}  \|_{L^2} ,
\]
and we notice that in view of \eqref{EQ-Fz}, the $L^2$ norms of $f_{ \tau + it}$ and $g_{ \tau + it}$
are bounded by constants independent of $t$. We now  estimate $\Norm{\sigma_z}{{L}^{\infty}}$. Let $E$ be
the set of all $(\xi_1,\dots , \xi_n) \in \mathbb{R}^n$ with some $\xi_i= 0$. Then for all
$(\xi_1,\dots , \xi_n) \in \mathbb{R}^n\setminus E$ there are only finitely many  indices $k_i$ in the
summation  defining $\sigma_z(\xi_1,\dots , \xi_n)$ that produce a nonzero term, in fact the indices with
$|\xi_i|/4\le 2^{k_i} \le 4 |\xi_i|$ for all $i\in \{1,\dots , n\}$. Also,
$ P_\rho(\tau+it) = P_\rho(\tau) + (s_1^1-s_1^0)( it)$, which implies that
\begin{align} \begin{split}\label{009988}
&\Gamma\big( -P_1(\tau+it),\dots,  -P_n(\tau+it) \big) \\
& \qq\qq =\Gamma\big(- P_1(\tau ),\dots,  -P_n(\tau ) \big) \,\,\Gamma\big( it(s_1^0-s_1^1) ,\dots, it(s_n^0-s_n^1) \big)\,.
\end{split}\end{align}
Applying identity \eqref{009988}, and using successively
Proposition~\ref{P:embedding},
Lemma~\ref{imaginary-Lorentz},
the fact that $ \textup{Re }  P_1(\tau+ it) =P_1 (\tau) $,  identity \eqref{BSha} together with  the fact
that  $h_{k_1,\dots , k_n}$ is measure-preserving, and Lemma~\ref{Lorentz maximal}, we estimate
$\Norm{\sigma_{\tau+it}}{{L}^{\infty}}$ by
\begin{align*}
& \sup_{\xi \in \mathbb{R}^n\setminus E}\!\!\! \sum_{ {\substack{ 1\le i \le n \\  \f{  |\xi_i| }4\le 2^{k_i} \le 4 |\xi_i| }}}\!\!\!
\bigg\Vert \Gamma\big(-P_1(\tau+it)  ,\dots, - P_n(\tau+it) \big) \left[ \varphi_{k_1,\dots, k_n} h_{k_1,\dots, k_n}^{s_1- P_1(\tau+it)} \right] \bigg\Vert_{L^{\infty} }\\
& \leq C\, (1+|t |)^n   \sup_{\xi \in \mathbb{R}^n\setminus E}\!\!\! \sum_{ {\substack{ 1\le i \le n \\  \f{  |\xi_i| }4\le 2^{k_i} \le 4 |\xi_i|   }}}\!\!\! \Norm{ \varphi_{k_1,\dots, k_n} h_{k_1,\dots, k_n}^{s_1-  {  P_1(\tau)  }  } }   { \Lambda_{\phi_{P_1(\tau) ,  (1-P_1(\tau))d}}
(\rn)} \\
& \leq C\,(1+|t |)^n \sup_{\xi \in \mathbb{R}^n\setminus E} \sum_{ {\substack{ 1\le i \le n \\
\f{  |\xi_i| }4\le 2^{k_i} \le 4 |\xi_i|}}}\Norm{ \varphi^*_{k_1,\dots, k_n}(r) r^{s_1-{ P_1(\tau)  }}}
{\Lambda_{\phi_{P_1(\tau), (1-P_1(\tau)) d}}(0,\infty) }  \\
& \leq C (1+|t |)^n \sup_{\xi \in \mathbb{R}^n\setminus E} \sum_{ {\substack{ 1\le i \le n \\  \f{  |\xi_i| }4\le 2^{k_i} \le 4 |\xi_i|}}}
\Norm{ \varphi^*_{k_1,\dots, k_n} } { \Lambda_{\phi_{s_1 ,  (1-P_1(\tau) )d}}  (0,\infty)} \\
& \leq  C\,  (1+|t |)^n 5^n \sup_{k_1,\dots , k_n\in \mathbb Z} \Norm{ \varphi_{k_1,\dots, k_n}}
{\Lambda_{\phi_{s_1, d}}(\rn)}\,,
\end{align*}
and the last expression is finite in view of assumption \eqref{assump77}. This proves that $|F(\tau+it)|\le A_{\tau}(t) $, where
$A_{\tau}(t)\le C' \, (1+|t|)^n$.
\end{proof}

To prove Theorem~\ref{1/2MainResult} we   apply Theorem~\ref{Interpolation} as follows: For the given $p$ with $1<p<2$
we set $p_0 =1+\epsilon$ for some small number $\epsilon$, and we define $\theta$ in terms of $(1-\theta)/p_0+\theta/2= 1/p$.

Given $0< s_1\le \dots \le s_n$ with exactly $d$ numbers among $s_2,\dots , s_n$ equal to $s_1$, pick
$\f12 < s_1^0\le \dots \le s_n^0$ and $0< s_1^1\le \dots \le s_n^1 \leq 1/2$ such that
$s_j = (1-\theta) s_j^0+\theta s_j^1$. This relationship maintains    proportions, and as the sequences are
all increasing, it must be the case that the first $d+1$ terms in each sequence are equal. We pick these sequences
so that $ s_1^0 =\dots =s_{d+1}^0=\frac12+ \varepsilon$   and   $ s_1^{1}=\dots=s_{d+1}^1$. We note that $s_1^1$
can be found thanks to the assumption $s_1>1/p-1/2$. Inequality~\eqref{E:first} follows from the special case
$1/2<s_1 <1$ of  Theorem~\ref{1/2MainResult} proved in Section~\ref{S:core}, while   inequality~\eqref{E:second} follows
from Proposition~\ref{P:embedding}.

\begin{Remark}\label{R:remark_logarithm}
Assume that all assumptions of Theorem~\ref{1/2MainResult} are satisfied and, in addition, $s_1\leq 1/2$. Let $\delta>0$.
We claim that inequality~\eqref{E:bound} holds with the (smaller) constant
\[
K=	\sup_{j_1,\dots , j_n\in \mathbb{Z} } \Norm{\Gamma \left(s_1,\dots , s_n\right)\big[  \wh{\Psi}\,
D_{j_1,\dots , j_n} \sigma \big]}{\Lambda_{\phi_{s_1, (1-s_1+\delta)  d}}(\mathbb{R}^n)}\,.
\]
This can be proved by employing a slight modification of the proof of Theorem~\ref{Interpolation}. Namely, we replace equation~\eqref{E:sigma} by
\begin{align*}
&\sigma_z  = \sum_{k_1,\dots, k_n \in \mathbb{Z}}D_{-k_1,\dots, -k_n} \\
&\bigg[  \widehat{\Psi_b} \, \Gamma\big(-P_1(z) ,\dots, - P_n(z) \big)
\Big[ \varphi_{k_1,\dots, k_n} h_{k_1,\dots, k_n}^{s_1-P_1(z)} (\log(e+h_{k_1,\dots,k_n}^{-1}))^{(P_1(z)-s_1)d}  \Big] \bigg]
\end{align*}
and define the function $F$ by~$\eqref{E:definition_F}$. Then one can show that
\begin{equation*}
|F(it)|  \leq CK_0(1+|t |)^n \sup_{j_1,\dots, j_n \in \mathbb{Z}} \Norm{\varphi_{j_1,\dots, j_n} }
{ \Lambda_{\phi_{s_1 , d(2s_1^0-s_1)}}       (\rn)}
\big( \!\Norm{f}{{L}^p }^p +\epsilon' \big)^{\frac{1}{p_0}}  \big( \! \Norm{g}{{L}^{p^{\prime}} }^{p^{\prime}}
+\epsilon' \big)^{\frac{1}{p^{\prime}_0}}
\end{equation*}
and
\begin{equation*}
|F(1+it)| \leq  CK_1(1+|t |)^n \sup_{j_1,\dots, j_n \in \mathbb{Z}}\Norm{\varphi_{j_1,\dots, j_n} }
{ \Lambda_{\phi_{s_1 , d(1-s_1)} }       (\rn)}  \big( \!\Norm{f}{{L}^{p} }^p +\epsilon' \big)^{\frac{1}{2}}
\big(\! \Norm{g}{{L}^{p^{\prime}} }^{p^{\prime}} +\epsilon' \big)^{\frac{1}{2}}\,.
\end{equation*}
This then implies
\[
\Norm{T_{\sigma}(f)}{{L}^p(\rn)} \leq C_* K_0^{1-\theta} K_1^{\theta}
\sup_{j_1,\dots, j_n \in \mathbb{Z}}\Norm{\varphi_{j_1,\dots, j_n} } { \Lambda_{\phi_{s_1 , d(2s_1^0-s_1)}} (\rn)} \Norm{f}{{L}^{p}(\rn)}\,.
\]
Choosing all parameters as in the proof of Theorem~$\ref{1/2MainResult}$ with $\epsilon<\delta/2$ yields the conclusion.
	
We also recall that if $s_1>1/2$ then Theorem~$\ref{1/2MainResult}$ holds with
\[
K=	\sup_{j_1,\dots , j_n\in \mathbb{Z} } \Norm{\Gamma \left(s_1,\dots , s_n\right)\big[  \wh{\Psi}\,
D_{j_1,\dots , j_n} \sigma \big]}{\Lambda_{\phi_{s_1, s_1 d}}(\mathbb{R}^n)};
\]
this was proved in Section~$\ref{S:core}$.
\end{Remark}

\end{document}